\documentclass[a4paper,USenglish]{lipics-v2021}
\usepackage{craigstyle}

\usepackage{tikz-cd} 
\usetikzlibrary{calc}
\usepackage{mathtools} 
\usepackage{mdframed} 

\newmdenv[
  innerleftmargin=1em,
  innerrightmargin=1em,
  leftmargin=4em,
  skipabove=1em,
  skipbelow=1em
]{property}

\newcommand{\tuple}[1]{\ensuremath{\langle{#1}\rangle}}
\newcommand{\set}[1]{\{#1\}}  
\newcommand{\e}{{\rm e}}
\newcommand{\zr}{{\rm f}}
\newcommand{\mt}{\land}
\newcommand{\jn}{\lor}
\newcommand{\ra}{\to}
\newcommand{\pd}{\cdot}
\newcommand{\iv}[1]{{#1}^{-1}}
\newcommand{\rd}{\slash}
\newcommand{\ld}{\backslash}
\newcommand{\eq}{\approx}

\newcommand{\m}[1]{{\bf {#1}}}
\newcommand{\prp}[1]{{\sf #1}}
\newcommand{\f}{\ensuremath{\varphi}}
\newcommand{\ps}{\ensuremath{\psi}}
\newcommand{\x}{\ensuremath{\chi}}
\renewcommand{\a}{\ensuremath{\alpha}}
\renewcommand{\b}{\ensuremath{\beta}}
\newcommand{\ga}{\ensuremath{\gamma}}
\newcommand{\si}{\ensuremath{\sigma}}
\newcommand{\de}{\ensuremath{\delta}}

\DeclareMathOperator{\Con}{\mathrm{Con}}

\newcommand{\cg}[1]{{\rm Cg}^{_{#1}}}

\newcommand{\der}[1]{\mathrel{\vdash_{\mkern-8mu\scriptscriptstyle\rule[-.5ex]{3pt}{0pt}#1}}}
\newcommand{\mdl}[1]{\mathrel{\vdash_{\mkern-8mu\scriptscriptstyle\rule[-.5ex]{3pt}{0pt}#1}}}
\newcommand{\lgc}[1]{\ensuremath{{\rm #1}}}
\newcommand{\De}{\Delta}
\newcommand{\Ga}{\Gamma}
\newcommand{\Si}{\Sigma}
\newcommand{\The}{\Theta}

\newcommand{\F}{\m{F}}
\newcommand{\Fm}{\m{Fm}}
\newcommand{\Fmc}{\mathrm{Fm}}
\newcommand{\Eqc}{\mathrm{Eq}}
\newcommand{\Fc}{\mathrm{F}}
\newcommand{\type}{\mathcal{L}}
\newcommand{\eps}{\ensuremath{\varepsilon}}
\newcommand{\N}{\mathbb{N}}
\newcommand{\V}{\mathcal{V}}
\newcommand{\Vfsi}{\mathcal{V}_{_\textup{FSI}}}

\newcommand{\Vsi}{\mathcal{V}_{_\text{SI}}}

\newcommand{\K}{\mathcal{K}}

\newcommand{\BLat}{\mathcal{BL}\mathit{at}}
\newcommand{\BDLat}{\mathcal{BDL}\mathit{at}}

\newcommand{\cmap}[1]{\ensuremath{\pi}_{#1}}
\newcommand{\mathrmL}{{\mathchoice{\mbox{\rm\L}}{\mbox{\rm\L}}{\mbox{\rm\scriptsize\L}}{\mbox{\rm\tiny\L}}}}
\newcommand{\xbar}{\overline{x}}
\newcommand{\ybar}{\overline{y}}
\newcommand{\zbar}{\overline{z}}

\title{Interpolation and Amalgamation}

\author{George Metcalfe}{Mathematical Institute, University of Bern}{george.metcalfe@unibe.ch}{}{}

\authorrunning{George Metcalfe}
\titlerunning{Interpolation and Amalgamation}

\begin{document}

\maketitle

\begin{abstract}
This chapter presents a state-of-the-art survey of relationships, traditionally referred to as `bridges', between interpolation properties for propositional logics --- including superintuitionistic, modal, and substructural logics --- and amalgamation properties for corresponding classes of algebraic structures. These bridges are developed in the framework of universal algebra and illustrated with a broad range of examples from logic and algebra, demonstrating their use in establishing properties for both fields.
\end{abstract}

\tableofcontents


\section{Introduction}\label{s:introduction}

The aim of this chapter is to establish and explore some remarkable relationships or `bridges' existing between various forms of interpolation for propositional logics and amalgamation for classes of algebraic structures. Such bridges have appeared regularly in the literature, both for specific families of logics --- notably, superintuitionistic, modal, and substructural logics (see,~e.g.,~\cite{Mak91,GM05,GJKO07,KO10}) --- and within the broader settings of abstract algebraic logic~\cite{Mad98,CP99} and model theory~\cite{Bac75}. Here we follow~\cite{Pig72,MMT14} in constructing these bridges in the framework of universal algebra, which supplies the tools required for a uniform presentation, while still covering the vast majority of propositional logics. To keep our presentation self-contained, we introduce the requisite elementary notions from universal algebra along the way, giving proofs of key theorems and providing pointers to the relevant literature where appropriate. Throughout, we incorporate a range of examples from logic and algebra that illustrate the usefulness of the `bridges' in both directions, and give references to more recent developments.

To get our bearings, let us consider first the historically pivotal case of intuitionistic propositional logic $\lgc{IPC}$, understood as a consequence relation $\der{\lgc{IPC}}$, defined syntactically by an axiom system or sequent calculus, or semantically via Heyting algebras or Kripke models. Maksimova proved in 1977 that precisely eight superintuitionistic logics (i.e., axiomatic extensions of $\lgc{IPC}$) have the Craig interpolation property (\prp{CIP}) or equivalently, in this setting, the deductive interpolation property  (\prp{DIP})~\cite{Mak77}. That is, for each such logic $\lgc{L}$ and formulas $\a$ and $\b$ satisfying $\a\der{\lgc{L}}\b$, there exists a formula $\ga$ satisfying $\a\der{\lgc{L}}\ga$ and $\ga\der{\lgc{L}}\b$,  whose variables occur in both $\a$ and $\b$. Maksimova's proof made use of Kripke semantics for superintuitionistic logics and was essentially algebraic. First, she proved that a superintuitionistic logic has the \prp{DIP} if, and only if, an associated variety (equational class) of Heyting algebras has the amalgamation property (\prp{AP}), and, second,  that there are precisely eight such varieties (see~\refchapter{chapter:nonclassical} for further details).  This result was later strengthened by Ghilardi and Zawadowski~\cite{GZ02}, who, building on Pitts' theorem for $\lgc{IPC}$~\cite{Pit92}, proved that these eight logics have the stronger property of uniform interpolation, and that the first-order theories of the eight varieties of Heyting algebras with the \prp{AP} have a model completion (see~\refchapter{chapter:uniform}).

Similar results have been established for families of normal modal logics~(see, e.g.,~\cite{CZ96,GM05}). In particular, Maksimova used appropriate amalgamation properties to prove that between forty-three and forty-nine axiomatic extensions of $\lgc{S4}$ have the \prp{DIP}, and between thirty-one and thirty-seven such extensions have the \prp{CIP}~\cite{Mak79,Mak91}, whereas continuum-many axiomatic extensions of G{\"o}del-L{\"o}b logic $\lgc{GL}$ have both properties~\cite{Mak91}.  Interpolation and amalgamation properties (and their failures) have also been established for diverse families of substructural logics and varieties of residuated algebras, respectively (see,~e.g.,~\cite{Mon06,GJKO07,KO10,Mar12,MMT14,GLT15,MPT23,JS25}). For example, precisely nine varieties of Sugihara monoids have the \prp{AP}, and hence precisely nine axiomatic extensions of the relevant logic $\lgc{RM_t}$ (R-Mingle with unit) have the \prp{DIP} or equivalently, in this case, the \prp{CIP}~\cite{MM12}. Moreover, a variety of MV-algebras has the \prp{AP} if, and only if, it is generated by a totally ordered MV-algebra, so countably infinitely many axiomatic extensions of \L ukasiewicz logic $\lgc{\mathrmL}$ have the \prp{DIP}~\cite{DL00}, despite the fact that classical propositional logic $\lgc{CPC}$ is the only consistent axiomatic extension of $\lgc{\mathrmL}$ that has the \prp{CIP}.  Conversely, syntactic proofs of the \prp{CIP} for substructural logics --- in particular, extensions of the Full Lambek Calculus with exchange --- have been used to establish the \prp{AP} for varieties of residuated algebras where no algebraic proof is known (see,~e.g.,~\cite{GJKO07,MPT23} for details).

Amalgamation emerged as a central concept in Schreier's work on amalgamated free products of groups in the 1920s~\cite{Sch27}. The \prp{AP} was formulated in full generality by Fra\"{\i}ss{\'e} in his 1954 paper~\cite{Fra54} and subsequently studied intensively  by J\'onsson in the setting of universal algebra~\cite{Jon56,Jon60,Jon61,Jon62,Jon65}. The relationship between amalgamation and interpolation was first considered by Daigneault in the context of polyadic algebras~\cite{Dai64}, and then, in a more general algebraic setting, by J\'onsson~\cite{Jon65}, but is credited in the latter to unpublished work of Keisler, and, as observed by Pigozzi in~\cite{Pig72}, essential ideas underlying the proof may be traced back to Magnus' work in group theory. The basis for this relationship for a variety of algebras (or corresponding logic) was identified explicitly in~\cite{Pig72}: equational consequence in the variety can be interpreted in terms of congruences of its free algebras. This interpretation provides the basis for a wealth of other bridge theorems between `algebraic' and `logical' properties studied in the literature (see, e.g.,~\cite{KMPT83,CP99,MMT14,FM24} for further examples and references). In particular, such a bridge theorem has been established between the Beth definability property for a propositional logic --- closely related to the \prp{CIP} in some contexts (see~\refchapter{chapter:propositional}) --- and surjectivity of epimorphisms in the corresponding variety~\cite{Nem84,BH06}.

Let us provide a brief overview of the contents of this chapter. In Section~\ref{s:algebra_and_logic}, we recall some basic notions of universal algebra and define equational consequence for a class of algebras, explaining in Section~\ref{s:consequence_and_congruence} how this definition can be recast for a variety in terms of congruences of its free algebras. In Section~\ref{s:cep}, we establish a bridge between the congruence extension property for a variety and the existence of an equational `local deduction theorem' known as the extension property, describing also how a distinguished subclass may suffice for checking these properties. Similarly, in Section~\ref{s:amalgamation}, we establish a bridge between the \prp{AP} for a variety (or, in some cases, a distinguished subclass) and a property of equational consequence known as the Robinson property. In Section~\ref{s:interpolation}, we use these bridge theorems to relate algebraic properties to (variants of) the \prp{DIP}. Finally, in Section~\ref{s:craig}, we turn our attention to the \prp{CIP} and corresponding superamalgamation property for varieties of algebras with a lattice reduct.


\section{Logic and Algebra}\label{s:algebra_and_logic}

In this section, we recall some basic tools of universal algebra (referring the reader to~\cite{BS81} for further details) and introduce the key notion of equational consequence for a class of algebras. We also explain, pointing to some familiar examples, how equational consequence provides an appropriate setting for relating propositional logics to a suitable `algebraic semantics'.

Let us begin by fixing an algebraic language $\type$ --- that is, a first-order language with no relation symbols --- and let $\type_n$ denote the set of operation (function) symbols of arity $n\in\mathbb{N}$. An {\em $\type$-algebra} $\m{A}$ consists of a non-empty set $A$ equipped with an operation $f^\m{A}\colon A^n\ra A$ for each $f\in\type_n$, typically written as $\tuple{A,f^\m{A}_1,\dots,f^\m{A}_k}$ when $\type$ has operation symbols $f_1,\dots,f_k$. An $\type$-algebra $\m{B}$ is a {\em subalgebra} of an $\type$-algebra $\m{A}$ if $B\subseteq A$ and $f^\m{B}(b_1,\dots,b_n)=f^\m{A}(b_1,\dots,b_n)$ for all  $f\in\type_n$ and $b_1,\dots,b_n\in B$.  The {\em direct product} $\m{A}=\prod_{i\in I}\m{A}_i$ of a family $\set{\m{A}_i}_{i\in I}$  of $\type$-algebras  consists of the Cartesian product $\prod_{i\in I} A_i$ equipped with an operation $f^{\m{A}}$ for each $f\in\type_n$ satisfying $f^{\m{A}}(a_1,\dots,a_n)(i)=f^{\m{A}_i}(a_1(i),\dots,a_n(i))$ for all  $i\in I$. 

Below we recall some classes of algebras that play a prominent role in the study of propositional logics, omitting sub- and superscripts when these are clear from the context.

\begin{example}
A {\em lattice} is defined order-theoretically as a partially ordered set $\tuple{L,\le}$ such that any  two elements $a,b\in L$ have a greatest lower bound $a\mt b$ and least upper bound $a \jn b$. It is said to be {\em bounded} if it has a greatest element $\top$ and least element $\bot$, and {\em complete} if every subset $A\subseteq L$ has a greatest lower bound $\bigwedge A$ and least upper bound $\bigvee A$. A lattice may also be defined  algebraically, however, in a language with binary operation symbols $\mt$ and $\jn$ as an algebra $\m{L}=\tuple{L,\mt,\jn}$ such that defining $a\le b :\Longleftrightarrow a\mt b=a$ yields a  lattice $\tuple{L,\le}$ in the order-theoretic sense. If $\tuple{L,\le}$ is bounded, then constant symbols $\bot,\top$ can be added to the language to obtain an algebra $\tuple{L,\mt,\jn,\bot,\top}$. A (bounded) lattice $\m{L}$ is called {\em distributive} if $a\mt (b\jn c)=(a\mt b)\jn(a\mt c)$ and  $a\jn (b\mt c)=(a\jn b)\mt(a\jn c)$ for all $a,b,c\in L$. \lipicsEnd
\end{example}

\begin{example}\label{e:Boolean}
A {\em Boolean algebra}, defined in the language of bounded lattices extended with a unary operation symbol $\lnot$, is an algebra $\m{B}=\tuple{B,\mt,\jn,\lnot,\bot,\top}$ such that $\tuple{B,\mt, \jn,\bot,\top}$ is a distributive bounded lattice and $\lnot a$ is the (necessarily unique) complement of $a \in B$, i.e., $a\mt\lnot a=\bot$ and $a\jn\lnot a=\top$. A {\em modal algebra} is defined in this language extended with a unary operation symbol $\Box$ as an algebra $\tuple{M,\mt,\jn,\lnot,\bot,\top,\Box}$ such that  $\tuple{M,\mt,\jn,\lnot,\bot,\top}$ is a Boolean algebra, $\Box\top=\top$, and $\Box(a\mt b)=\Box a\mt \Box b$ for all $a,b\in M$. \lipicsEnd
\end{example}

\begin{example}\label{e:Heyting}
A {\em Heyting algebra}, defined in the language of bounded lattices extended with a binary operation symbol $\ra$, is an algebra $\m{H}=\tuple{H,\mt,\jn,\ra,\bot,\top}$ such that $\tuple{H,\mt,\jn,\bot,\top}$ is a bounded distributive lattice and $\ra$ is the residual of $\mt$, i.e., $a\mt b\le c \iff a\le b\ra c$ for all $a,b,c\in H$. Boolean algebras are term-equivalent to Heyting algebras satisfying $\lnot\lnot a=a$ for all $a\in H$, where $\lnot a\coloneqq a\ra\bot$; that is, $\tuple{H,\mt,\jn,\lnot,\bot,\top}$ is a Boolean algebra and, conversely, for any Boolean algebra $\m{B}$, defining $a\ra b\coloneqq \lnot a\jn b$ yields a Heyting algebra $\tuple{B,\mt,\jn,\ra,\bot,\top}$ satisfying $\lnot\lnot a=a$ for all $a\in B$. Heyting algebras and Boolean algebras provide algebraic semantics for intuitionistic propositional logic $\lgc{IPC}$ and classical propositional logic $\lgc{CPC}$, respectively, and other classes of Heyting algebras play this role for superintuitionistic logics. In particular,  Heyting algebras satisfying $(a\ra b)\jn(b\ra a)=\top$ for all $a,b\in H$  provide algebraic semantics for G{\"o}del logic $\lgc{G}$ and are called {\em G{\"o}del algebras} (see,~e.g.,~\cite{CZ96}). \lipicsEnd
\end{example}

\begin{example}\label{e:FL-algebra}
An {\em FL-algebra} (or {\em pointed residuated lattice}), defined in a language with binary operation symbols $\mt$, $\jn$, $\pd$, $\ld$, $\rd$, and constant symbols $\zr$, $\e$, is an algebra $\m{L}=\tuple{L,\mt,\jn,\pd,\ld,\rd,\zr,\e}$ such that $\tuple{L,\mt,\jn}$ is a lattice, $\tuple{L,\cdot,\e}$ is a monoid, and $\ld$ and $\rd$ are residuals of $\pd$, i.e., $b \le a \ld c\iff a b \le c \iff a \le c \rd b$ for all $a,b,c\in L$. If $ab=ba$ for all $a,b\in L$, then $\m{L}$ is called an {\em FL$_\text{e}$-algebra} and we define $a\ra b\coloneqq a\ld b=b\rd a$. FL-algebras and FL$_\text{e}$-algebras provide algebraic semantics for the full Lambek calculus $\lgc{FL}$ and full Lambek calculus with exchange $\lgc{FL_e}$, respectively, and their subclasses play this role for a wide range of substructural logics  (see,~e.g.,~\cite{GJKO07,MPT23}). For example, Heyting algebras are term-equivalent to FL-algebras satisfying $a\mt b =ab$ and $\zr \le a$ for all $a,b\in L$, and {\em MV-algebras} --- algebraic semantics  for  \L ukasiewicz logic $\lgc{\mathrmL}$ --- are term-equivalent to FL$_\text{e}$-algebras satisfying $a\jn b = (a \ra b)\ra b$ and $\zr \le a$ for all $a,b\in L$. Let us note also that FL-algebras  satisfying $\e=\zr$ and $a(a\ld\e)=\e$ for all $a\in L$ are term-equivalent to {\em $\ell$-groups}, algebras with their own extensive literature (see,~e.g.,~\cite{AF88}) that play a key role in the structure theory of residuated algebras (see~\cite{MPT23}). \lipicsEnd
\end{example}

An important ingredient of the interplay between logic and algebra is the notion of a formula algebra over a set of variables. Let us denote arbitrary sets of variables by $\xbar,\ybar,\zbar$, assuming without further comment that these are disjoint, and denote unions by writing $\xbar,\ybar$.  We also assume for a more streamlined presentation that $\type$ has at least one constant symbol.\footnote{All definitions and theorems presented here are easily adjusted to accommodate languages with no constant symbols by adding assumptions that certain sets or their intersections  are non-empty.}  The  {\em $\type$-formula algebra $\Fm_\type(\xbar)$ over $\xbar$} (often referred to as a {\em term algebra}) consists of  the set $\Fmc_\type(\xbar)$ of {\em $\type$-formulas over~$\xbar$}, built inductively using $\xbar$ and the operation symbols of $\type$, with an operation $f^{\Fm_\type(\xbar)}$ for each $f\in\type_n$ that maps  $\a_1,\dots,\a_n\in\Fmc_\type(\xbar)$ to $f(\a_1,\dots,\a_n)$. Note that since, by assumption, $\type$ has at least one constant symbol, $\Fm_\type(\emptyset)$ always exists.

Although consequence in a propositional logic is typically defined over formulas, it is convenient when relating logic to algebra to also consider consequences between equations. Formally, an $\type$-{\em equation over $\xbar$} is an ordered pair of $\type$-formulas $\a,\b\in\Fmc_\type(\xbar)$, denoted by $\a\eq\b$, and the set of $\type$-equations over  $\xbar$ is identified with the set $\Eqc_\type(\xbar)\coloneqq(\Fmc_\type(\xbar))^2$. When the language is clear from the context, we drop the subscript $\type$.

A map $\f\colon A\ra B$ between  $\type$-algebras $\m{A}$ and $\m{B}$ is a {\em homomorphism}, denoted by writing $\f\colon\m{A}\ra\m{B}$, if for each $f\in\type_n$ and all $a_1,\ldots,a_n\in A$,
\[
\f (f^\m{A}(a_1,\ldots,a_n))=f^\m{B}(\f(a_1),\ldots,\f(a_n)),
\]
and its {\em kernel} is defined as $\ker(\f)\coloneqq\set{\tuple{a_1,a_2}\in A^2\mid\f(a_1)=\f(a_2)}$. 

An injective homomorphism $\f\colon\m{A}\to\m{B}$ is called an {\em embedding} of $\m{A}$ into $\m{B}$, and if $\f$ is bijective, it is called an  {\em isomorphism} between $\m{A}$ and $\m{B}$. If there exists a surjective homomorphism from $\m{A}$ to $\m{B}$, then $\m{B}$ is said to be a {\em homomorphic image} of $\m{A}$, and if there exists an isomorphism between $\m{A}$ and $\m{B}$, then $\m{A}$ is said to be {\em isomorphic} to $\m{B}$, denoted by writing~$\m{A}\cong\m{B}$. For convenience, we also extend a homomorphism $\f\colon\m{A}\ra\m{B}$ to the homomorphism $\f\colon\m{A}\times\m{A}\ra\m{B}\times\m{B};\:\tuple{a_1,a_2}\mapsto\tuple{\f(a_1),\f(a_2)}$.

We now have the  ingredients required to define equational consequence with respect to some given class $\K$ of $\type$-algebras. For any set $\xbar$ and $\Si\cup\set{\eps}\subseteq\Eqc(\xbar)$, let
\begin{align*}
\Si\mdl{\K}\eps\: :\Longleftrightarrow&\:\text{ for any $\m{A}\in\K$ and homomorphism $\f\colon{\Fm(\xbar)}\ra\m{A}$,}\\[.025in]
&\qquad \Si\subseteq\ker(\f) \enspace\Longrightarrow\enspace \eps\in\ker(\f),
\end{align*}
and for any $\Si\cup\De\subseteq\Eqc(\xbar)$,

\begin{itemize}
\item[] $\Si\mdl{\K}\De :\Longleftrightarrow\: \Si\mdl{\K}\eps$ for each $\eps\in\De$.\footnote{To confirm that $\Si\mdl{\K}\eps$ is well-defined, we should check that the defining condition is independent of the set $\xbar$ for which $\Si\cup\set{\eps}\subseteq\Eqc(\xbar)$. To this end, it suffices to observe that every homomorphism $\f\colon  {\Fm(\xbar)}\ra\m{A}$ extends to a homomorphism $\hat{\f}\colon{\Fm(\xbar,\ybar)}\ra\m{A}$ with $\ker(\hat{\f})\cap\Fmc(\xbar)^2=\ker(\f)$, and every homomorphism $\ps\colon  {\Fm(\xbar,\ybar)}\ra\m{A}$ restricts to a homomorphism $\ps'\colon  {\Fm(\xbar)}\ra\m{A}$ with $\ker(\ps)\cap\Fmc(\xbar)^2=\ker(\ps')$.}
\end{itemize}

\noindent
It is easily checked that this notion of consequence, when restricted to a fixed set $\xbar$, yields an (abstract) consequence relation over $\Eqc(\xbar)$; that is, for any $\Si\cup\Pi\cup\set{\eps}\subseteq\Eqc(\xbar)$,

\begin{itemize}
\item if $\eps\in\Si$, then $\Si\mdl{\K}\eps$\: {\em (reflexivity)};
\item if $\Si\mdl{\K}\eps$ and $\Si\subseteq\Pi$, then $\Pi\mdl{\K}\eps$\: {\em (monotonicity)};
\item if $\Si\mdl{\K}\eps$  and $\Pi\mdl{\K}\Si$, then $\Pi\mdl{\K}\eps$\: {\em (transitivity)}.
\end{itemize}

\noindent
These consequence relations are also {\em substitution-invariant}; that is, for any $\Si\cup\set{\eps}\subseteq\Eqc(\xbar)$ and 
homomorphism (substitution) $\si\colon\Fm(\xbar)\ra\Fm(\xbar)$,

\begin{itemize}
\item  if $\Si\mdl{\K}\eps$, then $\si[\Si]\mdl{\K}\si(\eps)$\:  {\em (substitution-invariance)}.
\end{itemize}

Certain classes of $\type$-algebras enjoy further useful properties and play an important role in universal algebra and (algebraic) logic. In particular, a class of $\type$-algebras is called a {\em variety} if it is closed under taking homomorphic images, subalgebras, and direct products. Equivalently, by a famous theorem of Birkhoff, a class of $\type$-algebras is a variety if, and only if, it is an {\em equational class}, i.e., a class of $\type$-algebras that satisfy some given set of $\type$-equations. Notably, equational consequence in a variety $\V$ is {\em finitary}; that is, for any $\Si\cup\set{\eps}\subseteq\Eqc(\xbar)$,

\begin{itemize}
\item  if $\Si\mdl{\V}\eps$, then $\Si'\mdl{\V}\eps$ for some finite $\Si'\subseteq\Si$\: {\em (finitarity)}.
\end{itemize}

\noindent
This statement can be deduced  using either the compactness theorem of first-order logic or  Lemma~\ref{l:bridge} below and the fact that the congruences of any algebra form an algebraic lattice.

A variety $\V$ of $\type$-algebras may serve as an algebraic semantics for a propositional logic~$\lgc{L}$, viewed as a substitution-invariant consequence relation $\der{\lgc{L}}$ defined over $\Fm(\xbar)$ for a countably infinite set $\xbar$. That is, there may exist {\em transformers} $\tau$ and $\rho$ that map formulas to finite sets of equations, and equations to finite sets of formulas, respectively, and satisfy for any $T\cup\set{\a}\subseteq\Fmc(\xbar)$ and $\Si\cup\set{\eps}\subseteq\Eqc(\xbar)$,

\begin{itemize}
\item $T\der{\lgc{L}}\a\iff\bigcup\tau[T]\mdl{\V}\tau(\a)$;
\item $\Si\mdl{\V}\eps\iff\bigcup\rho[\Si]\der{\lgc{L}}\rho(\eps)$;
\item $\set{\a}\der{\lgc{L}}\bigcup\rho[\tau(\a)]$\, and\, $\bigcup\rho[\tau(\a)]\der{\lgc{L}}\a$;
\item $\set{\eps}\mdl{\V}\bigcup\tau[\rho(\eps)]$\, and\, $\bigcup\tau[\rho(\eps)]\mdl{\V}\eps$,
\end{itemize}

\noindent
and, for any $\a\in\Fmc(\xbar)$, $\eps\in\Eqc(\xbar)$, and homomorphism (substitution) $\si\colon\Fm(\xbar)\ra\Fm(\xbar)$,

\begin{itemize}
\item $\si[\tau(\a)]=\tau(\si(\a))$\, and\, $\si[\rho(\eps)]=\rho(\si(\eps))$.
\end{itemize}

\noindent
When such transformers exist, they allow us to translate equational consequences into consequences between formulas, and vice versa. In particular, deductive interpolation and other syntactic properties may be interpreted as concerning either equations or formulas, depending on context and convenience. For further details, we refer the reader to the vast literature on abstract algebraic logic (see, e.g.,~\cite{BP89,Fon16}). Let us just note that we restrict our account here to varieties rather than quasivarieties (or more general classes of algebras), partly to avoid additional complexity and partly because varieties already provide algebraic semantics for the most well-studied non-classical logics.

\begin{example}
Heyting algebras form a variety that provides algebraic semantics for $\lgc{IPC}$ via transformers $\tau$, mapping a formula $\a$ to the set of equations $\set{\a\eq\top}$, and $\rho$, mapping an equation $\a\eq\b$ to the set of formulas $\set{\a\ra\b,\b\ra\a}$. By general results of abstract algebraic logic, each axiomatic extension of $\lgc{IPC}$  then has an algebraic semantics (via the same transformers) provided by the variety of Heyting algebras satisfying equations corresponding to the additional axioms. In particular, Boolean algebras and G{\"o}del algebras provide algebraic semantics for $\lgc{CPC}$ and $\lgc{G}$, respectively. These transformers can also be used to show that varieties of modal algebras serve as algebraic semantics for axiomatic extensions of $\lgc{K}$. \lipicsEnd
\end{example}

\begin{example}
FL-algebras form a variety that provides algebraic semantics for the full Lambek calculus $\lgc{FL}$ via transformers $\tau$, mapping a formula $\a$ to the set of equations $\set{\a\mt\e\eq\e}$, and $\rho$, mapping an equation $\a\eq\b$ to the set of formulas $\set{\a\ld\b,\b\ld\a}$  (see, e.g.,~\cite{GJKO07,MPT23}). Algebraic semantics for other well-known substructural logics are provided by various varieties of FL-algebras (see~\cite{GJKO07,MPT23}). In particular, FL$_\text{e}$-algebras and MV-algebras (see Example~\ref{e:FL-algebra}), provide algebraic semantics for $\lgc{FL_e}$ and {\L}ukasiewicz logic $\lgc{\mathrmL}$, respectively, and, although a logic of $\ell$-groups has not been considered,  Abelian $\ell$-groups (i.e., $\ell$-groups satisfying $xy\eq yx$) provide algebraic semantics for Abelian logic $\lgc{A}$ (see, e.g.,~\cite{MPT23}). 
\lipicsEnd
\end{example}

\begin{remark}
It is not hard to see that transformers restricted to single formulas and equations are available for the previous examples; just replace $\set{\a\ra\b,\b\ra\a}$ and $\set{\a\ld\b,\b\ld\a}$ with $\set{(\a\ra\b)\mt(\b\ra\a)}$ and $\set{(\a\ld\b)\mt(\b\ld\a)}$, respectively. However, for many logics --- including the implicational fragment of $\lgc{IPC}$ with an algebraic semantics provided by Hilbert algebras --- more than one formula or equation is necessary (see, e.g.,~\cite{Fon16} for details). \lipicsEnd
\end{remark}


\section{Consequence and Congruence}\label{s:consequence_and_congruence}

In this section, we recall some further elementary notions of universal algebra and present an interpretation of equational consequence in a variety in terms of the congruences of its free algebras (Lemma~\ref{l:bridge}). This result, which provides the scaffolding for proofs of the bridge theorems in subsequent sections, was first established explicitly in~\cite{Pig72}.

Consider any $\type$-algebra $\m{A}$, recalling that $\type$ is a fixed arbitrary algebraic language with at least one constant symbol.  A {\em congruence} of $\m{A}$ is an equivalence relation on $A$ satisfying for each $f\in\type_n$ and all $a_1,b_1,\dots,a_n,b_n\in A$,
\[
\tuple{a_1,b_1},\ldots,\tuple{a_n,b_n}\in\The\enspace\Longrightarrow\enspace
\tuple{f^\m{A}(a_1,\ldots,a_n),f^\m{A}(b_1,\ldots,b_n)}\in\The.
\]
Notably, the kernel of any homomorphism $\f\colon\m{A}\ra\m{B}$ is a congruence of $\m{A}$.

The set $\Con\m{A}$ of congruences of $\m{A}$ is closed under taking arbitrary intersections and hence forms a complete lattice $\tuple{\Con\m{A},\subseteq}$ with least element $\De_A\coloneqq\set{\tuple{a,a}\mid a\in A}$ and greatest element $A\times A$. Clearly, $\bigwedge S = \bigcap S$ for any  $S\subseteq\Con\m{A}$, but  $\bigcup S$ may not be a congruence, so $\bigvee S$ is the congruence of $\m{A}$ generated by $\bigcup S$, i.e., the smallest congruence of~$\m{A}$ containing~$\bigcup S$. More generally, the {\em congruence of $\m{A}$ generated by $R\subseteq A\times A$} is 
\[
\cg{\m{A}}(R)\coloneqq\bigcap\set{\The\in\Con\m{A}\mid R\subseteq\The}.
\]
A congruence $\The\in\Con\m{A}$ is called {\em finitely generated} if $\The=\cg{\m{A}}(R)$ for some finite $R\subseteq A\times A$.\footnote{The map $\cg{\m{A}}$ on $\mathcal{P}(A^2)$ is  an {\em algebraic closure operator} on $A \times A$, corresponding to the fact that $\tuple{\Con\m{A},\subseteq}$ is an {\em algebraic lattice} whose compact elements are the finitely generated congruences of $\m{A}$.}

Just as normal subgroups and ideals are used to construct quotient groups and rings, so congruences are used to construct quotients of arbitrary algebras. Given any $\type$-algebra $\m{A}$ and $\The\in\Con\m{A}$, the $\type$-algebra $\m{A}/\The$ consists of the set $A/\The$ of $\The$-equivalence classes of $A$ with a well-defined (since $\The$ is a congruence) operation $f^{\m{A}/\The}$ for each $f\in\type_n$ satisfying $f^{\m{A}/\The}(a_1/\The,\dots,a_n/\The)= f(a_1,\dots,a_n)/\The$ for all $a_1,\dots,a_n\in A$. Observe also that the {\em canonical map} $\cmap{\The}\colon A\ra A /\The;\; a\mapsto a/\The$ is a surjective homomorphism from $\m{A}$ onto $\m{A} /\The$ with $\ker(\cmap{\The})=\The$, so $\Con\m{A}$ consists of precisely  the kernels of homomorphisms $\f\colon\m{A}\to\m{B}$.

Congruences and quotients are used to formulate generalizations of the usual isomorphism theorems for groups and rings, including the following {\em general homomorphism theorem}.

\begin{theorem}[cf.~{\cite[Theorem~B.2]{MPT23}}]\label{t:general_homomorphism_theorem}
 For any surjective homomorphism $\pi\colon\m{A}\ra\m{B}$ and  homomorphism $\f\colon\m{A}\ra\m{C}$ with $\ker(\pi)\subseteq\ker(\f)$, there exists a unique homomorphism $\ps\colon\m{B}\ra\m{C}$ satisfying $\ps\pi=\f$; moreover, $\ps$ is injective if, and only if, $\ker(\pi)=\ker(\f)$. 
\end{theorem}

Now let $\V$ be any variety of $\type$-algebras and let $\xbar$ be any set. We obtain a congruence $\Phi_\V(\xbar)$ of the formula algebra $\Fm(\xbar)$ by defining for $\a,\b\in\Fmc(\xbar)$,
\[
\tuple{\a,\b}\in\Phi_\V(\xbar)\, :\Longleftrightarrow\enspace\mdl{\V}\a\eq\b.
\]
Equivalently, $\Phi_\V(\xbar)=\bigcap I_\V(\xbar)$, where $I_\V(\xbar)$ is the set $\set{\The\in\Con\Fm(\xbar)\mid\Fm(\xbar)/\The\in\V}$. The {\em free algebra of $\V$ over $\xbar$} may then be defined as the quotient algebra
\[
\F_\V(\xbar)\,\coloneqq\,\Fm(\xbar)  /\Phi_\V(\xbar).
\]
Let us drop the subscript $\V$ when the variety is clear from the context, and write $\a$ to denote both a formula $\a$ in $\Fm(\xbar)$ and its image $\a/\Phi(\xbar)$ under $\cmap{\Phi(\xbar)}$  in $\F(\xbar)$. Then $\F(\xbar)$ is generated by $\xbar$ and enjoys the universal mapping property for $\V$: every map from $\xbar$ to $\m{A}\in\V$ extends to a unique homomorphism from $\F(\xbar)$ to $\m{A}$. Observe also that the homomorphism $\f$ from $\Fm(\xbar)$ to the direct product $\prod\set{\Fm(\xbar)/\The\mid\The\in I(\xbar)}$ satisfying $\f(\a)(\The)=\cmap{\The}(\a)$ for each $\a\in\Fmc(\xbar)$ and all $\The\in I(\xbar)$, has kernel $\Phi(\xbar)$. From this observation and Theorem~\ref{t:general_homomorphism_theorem}, it follows that $\F(\xbar)$ embeds into a direct product of members of $\V$, and hence, since $\V$ is a variety, $\F(\xbar)\in\V$. Let us also assume for convenience, and without loss of generality, that $\F(\xbar)$ is a subalgebra of $\F(\xbar,\ybar)$ for any disjoint sets of variables $\xbar$, $\ybar$.

We now have all the ingredients necessary to state and prove the key lemma relating equational consequence in a variety to congruences of its free algebras.

\begin{lemma}\label{l:bridge}
For any variety $\V$ and $\Si\cup\set{\eps}\subseteq\Eqc(\xbar)$,
\[
\Si\mdl{\V}\eps\iff\eps\in\cg{\F(\xbar)}(\Si).
\]
\end{lemma}
\begin{proof}
Let $\pi$ denote the canonical map from $\Fm(\xbar)$ onto $\F(\xbar)$ and define $\The\coloneqq\cg{\F(\xbar)}(\pi[\Si])$, recalling that $\eps\in\cg{\F(\xbar)}(\Si)$ is notational shorthand for $\pi(\eps)\in\The$. 

Suppose first that $\Si\mdl{\V}\eps$ and consider the homomorphism $\f\colon\F(\xbar)\ra\F(\xbar)/\The;\:\a\mapsto\a/\The$. Since  $\F(\xbar)/\The\in\V$ and $\Si\subseteq\ker(\f\pi)$, by assumption, $\eps\in\ker(\f\pi)$. Hence $\pi(\eps)\in\ker(\f)=\The$. 

For the converse, suppose that $\pi(\eps)\in\The$ and consider any $\m{A}\in\V$ and homomorphism $\f\colon\Fm(\xbar)\ra\m{A}$ such that $\Si\subseteq\ker(\f)$. Since $\ker(\pi)=\Phi_\V(\xbar)\subseteq\ker(\f)$, there exists, by Theorem~\ref{t:general_homomorphism_theorem}, a homomorphism $\ps\colon\F(\xbar)\ra\m{A}$ satisfying $\ps\pi=\f$. It follows that $\pi[\Si]\subseteq\ker(\ps)$ and, by assumption, $\pi(\eps)\in\The\subseteq\ker(\ps)$. So $\eps\in\ker(\ps\pi)=\ker(\f)$. Hence $\Si\mdl{\V}\eps$.
 \end{proof}


\section{Local Deduction Theorems and the Congruence Extension Property}\label{s:cep}

In this section, we illustrate the usefulness of Lemma~\ref{l:bridge} by relating the well-known congruence extension property to a general `local deduction theorem' for equational consequence known as the extension property. This property was studied in~\cite{Ono86} as the `limited \prp{GINT}' and Theorem~\ref{t:cepiffep} below may be viewed as a refinement of Theorem~8 from this paper; it also appears in an abstract algebraic logic setting as the `extension interpolation property'  in~\cite{CP99}, and as one of the model-theoretic properties considered in~\cite{Bac75}.

Many propositional logics admit a `local deduction theorem', which allows a formula occurring as a premise in a consequence to be combined with the conclusion and vice versa. The following property provides a general formulation of such relationships in the setting of equational consequence for varieties.

\begin{property}
A variety $\V$ has the {\em extension property} (\prp{EP}) if for any $\Si\subseteq\Eqc(\xbar,\ybar)$ and $\Pi\cup\set{\eps}\subseteq\Eqc(\ybar)$ satisfying  $\Si\cup\Pi\mdl{\V}\eps$, there exists a $\De\subseteq\Eqc(\ybar)$ satisfying $\Si\mdl{\V}\De$ and $\De\cup\Pi\mdl{\V}\eps$.
\end{property}

\begin{example}\label{e:HeytingEP}
The fact that every variety of Heyting algebras has the $\prp{EP}$ is a direct consequence of the deduction theorem for superintuitionistic logics; that is, for any such logic $\lgc{L}$ and set of formulas $T\cup\set{\a,\b}$,
\[
T\cup\set{\a}\der{\lgc{L}}\b\iff T\der{\lgc{L}}\a\ra\b.
\]
Similarly, every variety of modal algebras has the $\prp{EP}$ by virtue of the local deduction theorem for normal modal logics; that is, for any such logic $\lgc{L}$ and set of formulas  $T\cup\set{\a,\b}$,
\[
T\cup\set{\a}\der{\lgc{L}}\b\iff T\der{\lgc{L}}(\a\mt\Box\a\mt \cdots\mt\Box^n\a) \ra\b\:\text{ for some }n\in\mathbb{N}.
\]
If  the variety of modal algebras satisfies $\Box x \le \Box\Box x$ (i.e., $\lgc{L}$ is an axiomatic extension of $\lgc{K4}$), then the right-hand side of the above equivalence can be simplified to $T\der{\lgc{L}}(\a\mt\Box\a)\ra\b$. \lipicsEnd
\end{example}

Let us show now, using the translation schema provided by Lemma~\ref{l:bridge}, that a variety $\V$ has the \prp{EP} if, and only if, for any $\The\in\Con\F(\xbar,\ybar)$ and $\Psi\in\Con\F(\ybar)$,
\begin{align}
(\cg{\F(\xbar,\ybar)}(\Psi)\jn\The)\cap {\Fc(\ybar)}^2=\Psi\jn(\The\cap {\Fc(\ybar)}^2).\label{id:ep}
\end{align}
Suppose first that $\V$ has the \prp{EP} and consider any $\The\in\Con\F(\xbar,\ybar)$ and $\Psi\in\Con\F(\ybar)$. For the non-trivial inclusion of \eqref{id:ep}, let $\eps\in (\cg{\F(\xbar,\ybar)}(\Psi)\jn\The)\cap {\Fc(\ybar)}^2$. Then $\The\cup\Psi\mdl{\V}\eps$ and, by the \prp{EP}, there  exists a $\De\subseteq\Eqc(\ybar)$ such that $\The\mdl{\V}\De$ and $\De\cup\Psi\mdl{\V}\eps$. Hence $\De\subseteq\The\cap\Fc(\ybar)^2$ and, as required, $\eps\in\Psi\jn(\The\cap\Fc(\ybar)^2)$. For the converse, given $\Si\subseteq\Eqc(\xbar,\ybar)$ and $\Pi\cup\set{\eps}\subseteq\Eqc(\ybar)$ satisfying $\Si\cup\Pi\mdl{\V}\eps$, let $\The\coloneqq\cg{\F(\xbar,\ybar)}(\Si)$ and $\Psi\coloneqq\cg{\F(\ybar)}(\Pi)$. Then applying \eqref{id:ep} and defining $\De\coloneqq\The\cap {\Fc(\ybar)}^2$ yields $\Si\mdl{\V}\De$ and $\De\cup\Pi\mdl{\V}\eps$.

We use this recasting of the \prp{EP} to relate it to the following well-known algebraic property.

\begin{property}
 A class $\K$ of $\type$-algebras  has the {\em congruence extension property} (\prp{CEP}) if for any $\m{B}\in\K$, subalgebra $\m{A}$ of $\m{B}$, and $\The\in\Con\m{A}$, there exists a $\Phi\in\Con\m{B}$ such that $\Phi\cap A^2=\The$. 
 \end{property}

\noindent
It is often convenient to use a slight reformulation of this property, observing that $\K$ has the \prp{CEP} if, and only if, $\cg{\m{B}}(\The)\cap A^2=\The$ for any $\m{B}\in\K$, subalgebra $\m{A}$ of $\m{B}$, and $\The\in\Con\m{A}$.

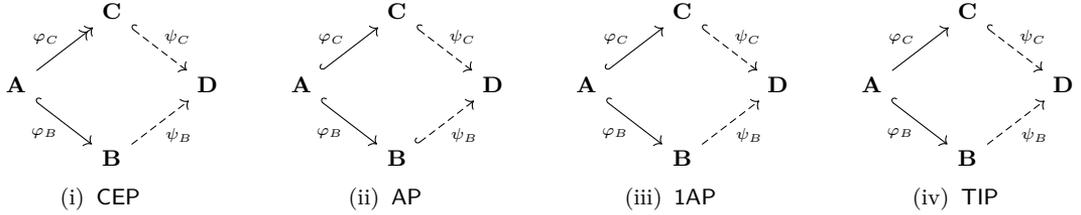
\begin{figure}[t]
\centering
\begin{footnotesize}
$\begin{array}{ccccccc}
\begin{tikzcd}
& \m{C} \ar[dashed, hook, rd, "\ps_C"] &       \\
\m{A} \ar[->>, ru, "\f_C"] \ar[hook, rd,"\f_B"'] &      & \m{D}  \\
 & \m{B} \ar[dashed, ru, "\ps_B"'] &     
\end{tikzcd}
&&
\begin{tikzcd}
& \m{C} \ar[dashed, hook, rd, "\ps_C"] &       \\
\m{A} \ar[hook, ru, "\f_C"] \ar[hook, rd,"\f_B"'] &      & \m{D}  \\
 & \m{B} \ar[dashed, hook, ru, "\ps_B"'] &     
\end{tikzcd}
& &
\begin{tikzcd}
& \m{C} \ar[dashed, hook, rd, "\ps_C"] &       \\
\m{A} \ar[hook, ru, "\f_C"] \ar[hook, rd,"\f_B"'] &      & \m{D}  \\
 & \m{B} \ar[dashed,  ru, "\ps_B"'] &     
\end{tikzcd}
& &
\begin{tikzcd}
& \m{C} \ar[dashed, hook, rd, "\ps_C"] &       \\
\m{A} \ar[ ru, "\f_C"] \ar[hook, rd,"\f_B"'] &      & \m{D}  \\
 & \m{B} \ar[dashed,  ru, "\ps_B"'] &     
\end{tikzcd}\\[.5in]
\text{(i)} \enspace \prp{CEP}  \quad& & \text{(ii)} \enspace\prp{AP} \quad & & \text{(iii)} \enspace\prp{1AP} \quad && \text{(iv)} \enspace\prp{TIP} \quad
\end{array}$
\end{footnotesize}
\caption{Commutative diagrams for algebraic properties}
\label{f:properties}
\end{figure}

Like many of the algebraic properties considered in this chapter, the \prp{CEP} admits an elegant presentation via commutative diagrams. A \emph{span} in a class of $\type$-algebras $\K$ is a 5-tuple $\tuple{\m{A},\m{B},\m{C},\f_B,\f_C}$ consisting of algebras $\m{A},\m{B},\m{C}\in \K$ and homomorphisms $\f_B\colon\m{A}\ra\m{B}$, $\f_C\colon\m{A}\ra\m{C}$. We call this span {\em injective} if $\f_B$ is an embedding, {\em doubly injective} if both $\f_B$ and $\f_C$ are embeddings, and {\em injective-surjective} if $\f_B$ is an embedding and $\f_C$ is surjective. 

We claim that a variety $\V$ has the \prp{CEP}  if, and only if, for any injective-surjective span $\tuple{\m{A},\m{B},\m{C},\f_B,\f_C}$ in $\V$, there exist a $\m{D}\in\V$, a homomorphism $\ps_B\colon\m{B}\ra\m{D}$, and an embedding $\ps_C\colon\m{C}\ra\m{D}$ such that $\ps_B \f_B=\ps_C\f_C$, that is, the diagram in Figure~\ref{f:properties}(i) is commutative. Suppose first that $\V$ has the $\prp{CEP}$. Let $\tuple{\m{A},\m{B},\m{C},\f_B,\f_C}$ be any injective-surjective span in $\V$, assuming without loss of generality that $\f_B$ is the inclusion map. Then  $\The\coloneqq\ker(\f_C)\in\Con\m{A}$ and Theorem~\ref{t:general_homomorphism_theorem} yields an isomorphism $\x\colon\m{A}/\The\ra\m{C}; \: a/\The\mapsto\f_C(a)$. Let $\Phi\coloneqq\cg{\m{B}}(\The)$ and $\m{D}\coloneqq\m{B}/\Phi$, noting that  $\Phi\cap A^2=\The$, by the \prp{CEP}. It follows that $\ps'_C\colon \m{A}/\The\ra\m{D};\: a/\The\mapsto a/\Phi$ is an embedding. Hence, defining $\ps_B\coloneqq\cmap{\Phi}$ and $\ps_C=\ps'_C\x^{-1}$, for any $a\in A$,
\[
\ps_B\f_B(a)=a/\Phi=\ps'_C(a/\The)=\ps'_C\x^{-1}\x(a/\The)=\ps_C\f_C(a).
\]
For the converse, suppose that the property holds, and consider any  subalgebra $\m{A}$ of $\m{B}\in\V$ and congruence $\The\in\Con\m{A}$. We obtain an injective-surjective span $\tuple{\m{A},\m{B},\m{A}/\The,\f_B,\f_C}$, where $\f_B$ is the inclusion map and $\f_C\coloneqq\cmap{\The}$. By assumption, there exist a $\m{D}\in\V$, a homomorphism $\ps_B\colon\m{B}\ra\m{D}$, and an embedding $\ps_C\colon\m{A}/\The\ra\m{D}$ such that $\ps_B \f_B=\ps_C\f_C$. Let $\Phi\coloneqq\ker(\ps_B)\in\Con\m{B}$. Then for any $\tuple{a_1,a_2}\in \Phi\cap A^2$,
\[
\ps_C(a_1/\The)=\ps_C\f_C(a_1)=\ps_B(a_1)=\ps_B(a_2)=\ps_C\f_C(a_2)= \ps_C(a_2/\The),
\]
and hence, by injectivity, $a_1/\The=a_2/\The$, i.e., $\tuple{a_1,a_2}\in\The$. So $\Phi\cap A^2=\The$ as required.

\begin{example}\label{e:groups}
Since the congruence lattice of a group $\m{G}=\tuple{G,\pd,\iv{},\e}$ is isomorphic to the lattice of its normal subgroups, $\m{G}$ has the \prp{CEP} if, and only if, for any subgroup $\m{H}$ of $\m{G}$ and normal subgroup $\m{N}$ of $\m{H}$, there exists a normal subgroup $\m{K}$ of $\m{G}$ satisfying $K\cap H=N$. Clearly, the variety of Abelian groups has the \prp{CEP}, since in this case every subgroup of $\m{G}$ is normal, and any subgroup of a subgroup of $\m{G}$ is also a subgroup of $\m{G}$. However, the variety of groups does not have the \prp{CEP}, since, for example, the alternating group $\m{A}_5$ is simple, but has subgroups that are not simple. \lipicsEnd
\end{example}

\begin{example}\label{e:latticesnocep}
The variety $\BLat$ of bounded lattices does not have the \prp{CEP}. Consider, for example, the bounded lattice $\m{M}_5=\tuple{\set{\bot,a,b,c,\top},\mt,\jn,\bot,\top}$ depicted  by the Hasse diagram:
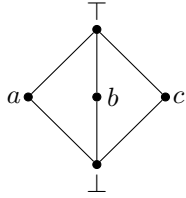
\begin{center}
\begin{tikzpicture}[scale=.9]
   \tikzstyle{every node}=[draw, circle, fill=black, minimum size=3pt, inner sep=0pt]
  \draw (0,0) node (bot) [label=below:\mbox{$\bot$}]{};
  \draw (0,2) node (top) [label=above:\mbox{$\top$}]{};
  \draw (-1,1) node (a) [label=left:\mbox{$a$}]{};
  \draw (0,1) node (b) [label=right:\mbox{$b$}]{};
  \draw (1,1) node (c) [label=right:\mbox{$c$}]{};
  \draw (bot) -- (b);
  \draw (top) -- (b); 
  \draw (bot) -- (a);
  \draw (top) -- (a);
  \draw (bot) -- (c);
  \draw (top) -- (c);
\end{tikzpicture}
\end{center}
Note that $\Con\m{M}_5=\set{\De_{M_5},(M_5)^2}$, i.e., $\m{M}_5$ is simple. 
Let $\m{A}$ be the sublattice of $\m{M}_5$ with $A=\set{\bot,a,c,\top}$ and let $\The$ be the congruence of $\m{A}$ with congruence classes $\set{a,\bot}$ and $\set{c,\top}$. Then $\De_{M_5}\cap A^2 =\De_{A}\neq\The$ and $(M_5)^2\cap A^2 = A^2\neq\The$, so $\m{M}_5$ does not have the \prp{CEP}. \lipicsEnd
\end{example}

We now establish the promised bridge between the \prp{CEP} and the \prp{EP}, noting that one direction implies that  varieties of algebras such as  Abelian groups enjoy a `local deduction theorem', while the other direction implies that varieties corresponding to propositional logics, such as Heyting algebras and modal algebras, possess a fundamental algebraic property.

\begin{theorem}\label{t:cepiffep}
A variety has the congruence extension property if, and only if, it has the extension property.
\end{theorem}
\begin{proof}
It suffices to prove that a variety $\V$ has the \prp{CEP} if, and only if, it satisfies~\eqref{id:ep}. For both directions, we will use the fact --- a consequence of the correspondence theorem for universal algebra (see, e.g.,~\cite[Theorem~B.4]{MPT23}) --- that for any $\type$-algebras $\m{B}$ and $\m{C}$, surjective homomorphism $\pi\colon\m{C}\ra\m{B}$, and $R\subseteq C\times C$,
\begin{align}
\pi^{-1}[\cg{\m{B}}(\pi[R])]=\cg{\m{C}}(R)\jn\ker(\pi).\label{id:corresp}
\end{align}
Suppose first that $\V$ has the \prp{CEP} and consider any $\The\in\Con\F(\xbar,\ybar)$ and $\Phi\in\Con\F(\ybar)$.  Let $\pi\colon\F(\xbar,\ybar)\ra\F(\xbar,\ybar)/\The$ be the canonical map with $\ker(\pi)=\The$  and let $\f$ be the restriction of $\pi$ to $\F(\ybar)$ with  $\ker(\f) =\The\cap\Fc(\ybar)^2$. Define $\m{B}\coloneqq\pi(\F(\xbar,\ybar))=\F(\xbar,\ybar)/\The$ and $\m{A}\coloneqq\f(\F(\ybar))$. Using \eqref{id:corresp} in the first and fifth steps and the \prp{CEP} for the third,
\begin{align*}
(\cg{\F(\xbar,\ybar)}(\Phi)\jn\The)\cap {\Fc(\ybar)}^2
& =\iv{\pi}[\cg{\m{B}}(\pi[\Phi])]\cap\Fc(\ybar)^2\\
& \subseteq\iv{\pi}[\cg{\m{B}}(\f[\Phi])\cap A^2]\\
& =\iv{\pi}[\cg{\m{A}}(\f[\Phi])]\\
& =\f^{-1}[\cg{\m{A}}(\f[\Phi])]\\
& = \Phi\jn (\The\cap\Fc(\ybar)^2)\\
& \subseteq (\cg{\F(\xbar,\ybar)}(\Phi)\jn\The)\cap {\Fc(\ybar)}^2.
\end{align*}
For the converse, suppose that $\V$ satisfies~\eqref{id:ep} and consider any congruence $\The$ of a subalgebra $\m{A}$ of some $\m{B}\in\V$. Let $\ybar\coloneqq A$ and $\xbar\coloneqq B{\setminus}A$, noting that $\The\subseteq A^2\subseteq\Fc(\ybar)^2\subseteq\Fc(\xbar,\ybar)^2$. Consider the surjective homomorphisms $\f\colon\F(\ybar)\ra\m{A}$ and $\pi\colon\F(\xbar,\ybar)\ra\m{B}$ extending the identity maps on $\ybar$ and $\xbar,\ybar$, respectively. Define $\Phi\coloneqq\ker(\pi)$, noting that $\ker(\f)=\Phi\cap\Fc(\ybar)^2$. Using  \eqref{id:corresp} in the first and fifth steps and \eqref{id:ep} in the third, 
\begin{align*}
\cg{\m{B}}(\The)\cap A^2
&=\pi[(\cg{\F(\xbar,\ybar)}(\The)\jn\Phi]\cap A^2\\
&\subseteq\pi[(\cg{\F(\xbar,\ybar)}(\The)\jn\Phi)\cap\Fc(\ybar)^2]\\
&=\pi[\cg{\F(\ybar)}(\The)\jn (\Phi\cap\Fc(\ybar)^2)]\\
&=\f[\cg{\F(\ybar)}(\The)\jn (\Phi\cap\Fc(\ybar)^2)]\\
&=\cg{\m{A}}(\f[\The])\\
&=\The\\
&\subseteq\cg{\m{B}}(\The)\cap A^2. \qedhere
\end{align*}
\end{proof}

\begin{example}\label{ex:FLeEP}
A description of the generation of congruences in a variety can be used to establish an explicit version of the \prp{EP}, typically described as a `local deduction theorem'. In particular, every variety of FL$_\text{e}$-algebras $\V$ has the \prp{CEP} and therefore the \prp{EP}, but also, more concretely (see~\cite{MPT23} for details), for any $\Si\subseteq\Eqc(\xbar)$ and $\a,\b\in\Fmc(\xbar)$,
\begin{align*}
\Si\cup\set{\e\le\a}\mdl{\V}\e\le\b  \:\iff\: \Si\mdl{\V} (\a\mt\e)^n\le\b\:\text{ for some } n\in\N.
\end{align*}
Equivalently, if $\lgc{L}$ is an axiomatic extension of $\lgc{FL_e}$, then for any $T\cup\set{\a,\b}\subseteq\Fmc(\xbar)$,
\begin{align*}
T\cup\set{\a}\der{\lgc{L}}\b  \:\iff\: T\mdl{\V} \a^n\ra\b\:\text{ for some } n\in\N,
\end{align*}
which for axiomatic extensions of $\lgc{IPC}$ simplifies to the familiar deduction theorem with $n=1$. Note, however, that the variety of FL-algebras does not have the \prp{CEP} (see, e.g.,~\cite[p.~217]{GJKO07}), so $\lgc{FL}$ does not admit a local deduction theorem of this form. \lipicsEnd
\end{example}

A significant obstacle to establishing an algebraic property such as the \prp{CEP} for a variety is the fact that  in principle it should be established for {\em all} its members. There exist, however, `transfer' results in the literature, however, that reduce such problems to more manageable subclasses, specifically the (finitely) subdirectly irreducible members that serve as  `building blocks' for all  members of the variety. 

A  {\em subdirect product} of a family of $\type$-algebras $\set{\m{B}_i}_{i\in I}$ is a subalgebra $\m{A}$ of $\prod_{i\in I}\m{B}_i$ such that  the projection map $\pi_i\colon\m{A}\to\m{B}_i;\: a\mapsto a(i)$ is surjective for each $i\in I$. An $\type$-algebra $\m{A}$ is  {\em (finitely) subdirectly irreducible} if for any isomorphism $\f$ between $\m{A}$ and a subdirect product of a (non-empty finite) family of $\type$-algebras $\set{\m{B}_i}_{i\in I}$, there is an $i\in I$ such that $\pi_i\f$ is an isomorphism.  Equivalently, an $\type$-algebra $\m{A}$ is subdirectly irreducible if $\De_A$ is completely meet-irreducible in $\Con\m{A}$ and finitely subdirectly irreducible if $\De_A$ is meet-irreducible in $\Con\m{A}$.\footnote{An element $a$ of a lattice $\m{L}$ is {\em meet-irreducible} if $a=b\mt c$ implies $a=b$ or $a=c$, and this is true of the greatest element $\top$ of $\m{L}$ if it has one; however, $a$ is {\em completely meet-irreducible} if $a=\bigwedge B$ implies $a\in B$ for any $B\subseteq L$, which is not the case for $\top=\bigwedge\emptyset$. In particular, we assume here that trivial algebras are finitely subdirectly irreducible  (following, e.g.,~\cite{CD90}) but not subdirectly irreducible.} By Birkhoff's subdirect representation theorem  (see, e.g.,~\cite[Theorem~B.7]{MPT23}), every $\type$-algebra is isomorphic to a subdirect product of subdirectly irreducible $\type$-algebras.

Let $\Vsi$ and $\Vfsi$ denote the classes of subdirectly irreducible and finitely subdirectly irreducible members of $\V$, respectively. Under certain conditions, properties such as the \prp{CEP} transfer from $\Vsi$ or $\Vfsi$ to $\V$ and, in some cases, back again. Often it is easier to consider the larger class $\Vfsi$. In particular, if $\V$ has equationally definable principal congruence meets (a common property for the algebraic semantics of a propositional logic that corresponds to having a suitable disjunction connective), then $\Vfsi$ is a universal class~\cite[Theorem~2.3]{CD90}.  

To obtain a transfer theorem for the \prp{CEP}, we require that $\V$ be {\em congruence-distributive}, that is, $\Con\m{A}$ should be distributive for every $\m{A}\in\V$. Since any $\type$-algebra with a lattice reduct is congruence-distributive, this requirement is fulfilled by the algebraic semantics of broad families of propositional logics.

\begin{theorem}[{\cite[Corollary~2.4]{FM24}}]\label{t:CEPvar}
Let $\V$ be any congruence-distributive variety.  Then $\V$ has the congruence extension property if, and only if, $\Vfsi$ has the congruence extension property. 
\end{theorem}

\begin{example}\label{e:CEPforBDL}
Theorem~\ref{t:CEPvar}  can drastically reduce the amount of work needed to check if a variety has the \prp{CEP}.  For example, the variety $\BDLat$ of bounded distributive lattices is congruence-distributive and $\BDLat_{_\textup{FSI}}$ contains,  up to isomorphism, only the trivial and two-element bounded lattices. Since $\BDLat_{_\textup{FSI}}$ clearly has the \prp{CEP}, so does $\BDLat$. \lipicsEnd
\end{example}

\begin{remark}
For a congruence-distributive variety $\V$, each member of $\Vfsi$ embeds into an ultraproduct of members of $\Vsi$~\cite[Lemma~1.5]{CD90}. Theorem~\ref{t:CEPvar} therefore implies that if  $\V$ is a congruence-distributive variety and $\Vsi$ is an elementary class, then $\V$ has the $\prp{CEP}$ if, and only if, $\Vsi$ has the $\prp{CEP}$. The latter was first proved in~\cite[Theorem~3.3]{Dav77} and follows also from a similar, but seemingly distinct, result for congruence-modular varieties~\cite[Theorem~2.3]{Kis85b}. \lipicsEnd
\end{remark}


\section{Amalgamation and the Robinson Property}\label{s:amalgamation}

In this section, we construct a bridge for varieties between the amalgamation property and a property of equational consequence known as the Robinson property, first established in~\cite{Pig72}.  

Let us fix again an algebraic language $\type$ that has at least one constant symbol and let $\K$ and $\K'$ be any classes of $\type$-algebras.  An {\em amalgam} in $\K'$ of a doubly injective span $\tuple{\m{A},\m{B},\m{C},\f_B,\f_C}$ in $\K$ is a triple $\tuple{\m{D},\ps_B,\ps_C}$ consisting of an algebra $\m{D}\in\K'$ and embeddings $\ps_B\colon\m{B}\ra\m{D}$ and $\ps_C\colon\m{C}\ra\m{D}$ such that $\ps_B\f_B = \ps_C\f_C$.

\begin{property}
A class $\K$ of $\type$-algebras has the {\em amalgamation property} (\prp{AP}) if every doubly injective span in $\K$ has an amalgam in $\K$ (see Figure~\ref{f:properties}(ii)). 
\end{property}

\noindent
Let us note also in passing that a class $\K$ of $\type$-algebras has the {\em strong amalgamation property} (strong \prp{AP}) if every doubly injective span $\tuple{\m{A},\m{B},\m{C},\f_B,\f_C}$ in $\K$ has an amalgam $\tuple{\m{D},\ps_B,\ps_C}$ in $\K$ satisfying $\ps_B\f_B[A] =\ps_B[B]\cap\ps_C[C]=\ps_C\f_C[A]$.

\begin{example}\label{e:BLatAP}
The variety $\BLat$ of bounded lattices has the strong \prp{AP}~\cite{Jon56}. Consider any doubly injective span $\tuple{\m{A},\m{B},\m{C},\f_B,\f_C}$ in $\BLat$,  assuming without loss of generality that $\f_B$ and $\f_C$ are inclusion maps and $A = B\cap C$. Let $R\coloneqq{\le^\m{B}\cup\le^\m{C}}$ and define for $x,y\in B\cup C$,
\begin{align*}
x\preceq y\,:\Longleftrightarrow\: Rxy\;\text{ or }\; (Rxz \text{ and }Rzy\text{, for some }z\in B\cap C).
\end{align*}
Then $\preceq$ is the smallest partial order on $B\cup C$ extending $\le^\m{B}$ and $\le^\m{C}$. Now let $\m{D}$ be the Dedekind-MacNeille completion of the poset $\tuple{B\cup C,\preceq}$, i.e., the set of subsets $X$ of $B\cup C$ satisfying $(X^u)^l = X$, ordered by set-inclusion, where $Y^l$ and $Y^u$ denote the sets of lower bounds and upper bounds of $Y\subseteq B\cup C$, respectively. Then the maps $\ps_B$ and $\ps_C$ sending an element $x$ to $\set{x}^l$ in $\tuple{B\cup C,\preceq}$ are embeddings of $\m{B}$ and $\m{C}$, respectively, into $\m{D}$, and satisfy $\ps_B(x) =\ps_C(x)$, for each $x\in A$. Moreover, since any element in $\ps_B[B]\cap\ps_C[C]$ is of the form $\set{b}^l=\ps_B(b)=\ps_C(c)= \set{c}^l$ for some $b\in B$ and $c\in C$, we obtain $b=c\in A$. Hence  $\ps_B\f_B[A] =\ps_B[B]\cap\ps_C[C]=\ps_C\f_C[A]$. \lipicsEnd
\end{example}

\begin{example}\label{e:HeytingAP}
The method described in Example~\ref{e:BLatAP} is easily adapted to establish that the variety of bounded semilattices has the strong \prp{AP}~\cite{Fle76}, and, with considerably more effort, can then be used to prove that the varieties of implicative semilattices and Heyting algebras have this property~\cite{Fle80}. The fact that the variety of Heyting algebras has the strong \prp{AP} was first proved in~\cite{Day??} and an alternative categorical proof may be found in~\cite{Pit83}.  \lipicsEnd
\end{example}

\begin{example}\label{e:monoidsAP}
Schreier's work on free amalgamated products implies that the variety of groups has the strong \prp{AP}~\cite{Sch27}, and it is not hard to see that the variety of Abelian groups also has this property. On the other hand, the varieties of monoids and commutative monoids do not even have the \prp{AP}. For a counterexample (adapted from~\cite{Kim57}), consider the commutative monoids $\m{A}$, $\m{B}$, and $\m{C}$  with $A=\set{u,v,w,0,\e}$, $B=A\cup\set{b}$, and $C=A\cup\set{c}$, where $\e$ is the neutral element, $bu=ub=v$, $cv=vc=w$, and all other products are $0$. If the doubly injective span $\tuple{\m{A},\m{B},\m{C},\f_B,\f_C}$, with inclusion maps $\f_B$, $\f_C$, were to have an amalgam $\tuple{\m{D},\ps_B,\ps_C}$ in the variety of monoids, then $\ps_C(w)=\ps_C(v)\ps_C(c)=\ps_B(v)\ps_C(c)=\ps_B(b)\ps_B(u)\ps_C(c)=\ps_B(b)\ps_C(u)\ps_C(c)=\ps_B(b)\ps_C(0)=\ps_B(b)\ps_B(0)=\ps_B(0)=\ps_C(0)$, contradicting $w\neq 0$.  \lipicsEnd
\end{example}

We now introduce the relevant property of equational consequence for varieties. 

\begin{property}
A variety $\V$ has the {\em Robinson property} (\prp{RP}) if for any $\Si\subseteq\Eqc(\xbar,\ybar)$ and $\Pi\subseteq\Eqc(\ybar,\zbar)$ satisfying $\Si\mdl{\V}\de\iff\Pi\mdl{\V}\de$ for all $\de\in\Eqc(\ybar)$, it follows that $\Si\cup\Pi\mdl{\V}\eps \iff\Pi\mdl{\V}\eps$ for any $\eps\in\Eqc(\ybar,\zbar)$.
\end{property}

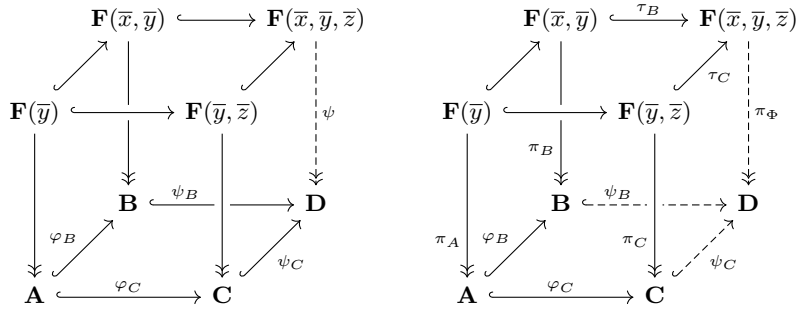
\begin{figure}[t]
\centering
\begin{small}
$\begin{array}{ccc}
\begin{tikzcd}[row sep={35,between origins}, column sep={35,between origins}]
 &\F(\xbar,\ybar)\ar[rr, hook]\ar[dd, ->>] & &\F(\xbar,\ybar,\zbar)\ar[dd, ->>, dashed, "\ps"]\\
\F(\ybar)\ar[ru, hook]\ar[rr, hook, crossing over]\ar[dd, ->>] &&\F(\ybar,\zbar)\ar[ru, hook]& \\
 &\m{B}\ar[rr, hook, "\ps_B" near start] &&\m{D}\\
\m{A}\ar[ru, hook, "\f_B"]\ar[rr, hook, "\f_C"] &&\m{C}\ar[ru, hook, "\ps_C" swap]\ar[from=uu, ->>, crossing over] &
\end{tikzcd}
&&
\begin{tikzcd}[row sep={35,between origins}, column sep={35,between origins}]
 &\F(\xbar,\ybar)\ar[rr, hook, "\tau_B"]\ar[dd, ->>, "\pi_B" swap, near end] & &\F(\xbar,\ybar,\zbar)\ar[dd, ->>, dashed, "\cmap{\Phi}"]\\
\F(\ybar)\ar[ru, hook]\ar[rr, hook, crossing over]\ar[dd, ->>, "\pi_A" swap, near end] &&\F(\ybar,\zbar)\ar[ru, hook, "\tau_C" swap]& \\
 &\m{B}\ar[rr, dashed, hook, "\ps_B" near start] &&\m{D}\\
\m{A}\ar[ru, hook, "\f_B"]\ar[rr, hook, "\f_C"] &&\m{C}\ar[ru, dashed, hook, "\ps_C" swap]\ar[from=uu, ->>, crossing over, "\pi_C" swap, near end] &
\end{tikzcd}
\end{array}$
\end{small}
\caption{Commutative diagrams for the proof of Theorem~\ref{t:AP<=>RP}}
\label{f:comdiag}
\end{figure}

\begin{theorem}\label{t:AP<=>RP}
A variety $\V$ has the amalgamation property if, and only if, it has the Robinson property.
\end{theorem}
\begin{proof}
Observe first, using Lemma~\ref{l:bridge} to translate between sets of equations and congruences of free algebras, that $\V$ has the \prp{RP}  if, and only if, for any $\The \in \Con \F(\xbar,\ybar)$ and $\Psi \in \Con \F(\ybar,\zbar)$ satisfying $\The \cap \Fc(\ybar)^2 = \Psi \cap \Fc(\ybar)^2$, there exists a $\Phi \in \Con \F(\xbar,\ybar,\zbar)$  satisfying $\The = \Phi \cap \Fc(\xbar,\ybar)^2$ and $\Psi = \Phi \cap \Fc(\ybar,\zbar)^2$. 

For the left-to-right direction, suppose that $\V$ has the \prp{AP} and consider any $\The \in \Con \F(\xbar,\ybar)$ and $\Psi \in \Con \F(\ybar,\zbar)$ satisfying $\The_0\coloneqq\The \cap \Fc(\ybar)^2 = \Psi \cap \Fc(\ybar)^2$. Define $\m{A}\coloneqq\F(\ybar)/\The_0$, $\m{B}\coloneqq\F(\xbar,\ybar)/\The$, and $\m{C}\coloneqq\F(\ybar,\zbar)/\Psi$. It follows that $\f_B\colon\m{A}\ra\m{B};\: [\a]_{\The_0}\mapsto [\a]_{\The}$ and $\f_C\colon\m{A}\ra\m{C};\: [\a]_{\The_0}\mapsto [\a]_{\Psi}$ are embeddings, and hence, by assumption, $\tuple{\m{A},\m{B},\m{C},\f_B,\f_C}$ has an amalgam $\tuple{\m{D},\ps_B,\ps_C}$ in $\V$. Moreover, we may assume without loss of generality that $\m{D}$ is generated by $\ps_B[B]\cup\ps_C[C]$. 

Let $\ps\colon\F(\xbar,\ybar,\zbar)\ra\m{D}$ be the unique surjective homomorphism that maps each $x\in\xbar$ to $\ps_B([x]_\The)$, each $y\in\ybar$ to $\ps_B([y]_\The)=\ps_C([y]_\Psi)$, and each $z\in\zbar$ to $\ps_C([z]_\Psi)$, as illustrated in the leftmost diagram of Figure~\ref{f:comdiag}. 
Let $\Phi\coloneqq\ker(\ps)$. We claim that $\The = \Phi \cap \Fc(\xbar,\ybar)^2$ and $\Psi = \Phi \cap \Fc(\ybar,\zbar)^2$, proving just the non-trivial inclusion of the first equality. Let $\tuple{\a,\b}\in\Phi\cap\Fc(\xbar,\ybar)^2$. Since $\ps$ is determined by the prescribed values for the generators of $\F(\xbar,\ybar,\zbar)$, clearly $\ps(\a)=\ps_B([\a]_{\The})$ and $\ps(\b)=\ps_B([\b]_{\The})$. But $\tuple{\a,\b}\in\Phi=\ker(\ps)$, so $\ps_B([\a]_{\The})=\ps(\a)=\ps(\b)=\ps_B([\b]_{\The})$, and,  since $\ps_B$ is injective, $[\a]_{\The}=[\b]_{\The}$ and $\tuple{\a,\b}\in\The$. 

For the right-to-left direction, suppose that $\V$ satisfies the above condition and consider any doubly injective span $\tuple{\m{A},\m{B},\m{C},\f_B,\f_C}$ in $\V$, assuming without loss of generality that $\f_B$ and $\f_C$ are inclusion maps with $B=\xbar,\ybar$, $C=\ybar,\zbar$, and $A=B\cap C=\ybar$. Extending identity maps, we obtain surjective homomorphisms $\pi_A\colon\F(\ybar)\ra\m{A}$, $\pi_B\colon\F(\xbar,\ybar)\ra\m{B}$, and $\pi_C\colon\F(\ybar,\zbar)\ra\m{C}$. Define $\The_0\coloneqq\ker(\pi_A)$, $\The\coloneqq\ker(\pi_B)$, and $\Psi\coloneqq\ker(\pi_C)$, observing that $\The_0=\The\cap\Fc(\ybar)^2=\Psi\cap\Fc(\ybar)^2$. By assumption, there exists a $\Phi\in\Con\F(\xbar,\ybar,\zbar)$ such that $\The=\Phi\cap\Fc(\xbar,\ybar)^2$ and $\Psi=\Phi\cap\Fc(\ybar,\zbar)^2$.  Define $\m{D}\coloneqq\F(\xbar,\ybar,\zbar)/\Phi$ and let $\cmap{\Phi}\colon\F(\xbar,\ybar,\zbar)\ra\m{D};\;\a\mapsto [\a]_\Phi$ be the canonical map and $\tau_B\colon\F(\xbar,\ybar)\ra\F(\xbar,\ybar,\zbar)$ and $\tau_C\colon\F(\ybar,\zbar)\ra\F(\xbar,\ybar,\zbar)$ be inclusion maps. Then $\ker(\cmap{\Phi}\tau_B)=\ker(\pi_B)$ and $\ker(\cmap{\Phi}\tau_C)=\ker(\pi_C)$. Hence, by Theorem~\ref{t:general_homomorphism_theorem}, there exist embeddings $\ps_B\colon\m{B}\ra\m{D}$ and $\ps_C\colon \m{C}\ra\m{D}$ such that $\ps_B\pi_B=\cmap{\Phi}\tau_B$ and $\ps_C\pi_C=\cmap{\Phi}\tau_C$, as illustrated by the rightmost diagram of Figure~\ref{f:comdiag}.  So $\ps_B\f_B =\ps_C\f_C$ and $\tuple{\m{D},\ps_B,\ps_C}$ is an amalgam of $\tuple{\m{A},\m{B},\m{C},\f_B,\f_C}$ in $\V$. 
\end{proof}

In order to obtain a precise match between amalgamation in a variety $\V$ and the subclass $\Vfsi$, we consider a property that is, in some contexts at least, weaker than the \prp{AP}.

\begin{property}
A class $\K$ of $\type$-algebras has the {\em one-sided amalgamation property} (\prp{1AP}) if  for any doubly injective span $\tuple{\m{A},\m{B},\m{C},\f_B,\f_C}$ in $\K$, there exist a $\m{D}\in\K$,  a homomorphism $\ps_B\colon\m{B}\ra\m{D}$, and an embedding $\ps_C\colon\m{C}\ra\m{D}$ such that $\ps_B \f_B=\ps_C\f_C$  (see Figure~\ref{f:properties}(iii)).
\end{property}

\noindent
In fact, a variety $\V$ has the \prp{1AP} if, and only, if it has the \prp{AP}. For the non-trivial direction, we apply the \prp{1AP} to a doubly injective span $\tuple{\m{A},\m{B},\m{C},\f_B,\f_C}$ and then again to the doubly injective span $\tuple{\m{A},\m{C},\m{B},\f_C,\f_B}$ to obtain $\m{D}_1,\m{D}_2\in\V$ with appropriate homomorphisms and embeddings, and obtain an amalgam $\m{D}_1\mathop{\times}\m{D}_2\in\V$ equipped with the induced embeddings. 

We now state a useful transfer theorem for the \prp{AP} and \prp{1AP}, recalling that many varieties that serve as algebraic semantics for propositional logics have both the \prp{CEP} --- e.g., via a local deduction theorem --- and a class of finitely subdirectly irreducible algebras that is closed under taking subalgebras --- often by virtue of having a suitable disjunction connective.

\begin{theorem}[{\cite[Corollary~3.5]{FM24}}]\label{t:APvar}
Let $\V$ be any variety with the congruence extension property such that $\Vfsi$ is closed under taking subalgebras. Then $\V$ has the amalgamation property if, and only if, $\Vfsi$ has the one-sided amalgamation property.
\end{theorem}

\noindent
This transfer theorem, which extends results in~\cite{GL71,MMT14} relating amalgamation in $\V$ to amalgamation in $\Vsi$, is particularly useful when investigating amalgamation in a variety $\V$ where the members of $\Vfsi$ have some simple structural features.

\begin{example}
Recall from Example~\ref{e:CEPforBDL} that $\BDLat$ has the \prp{CEP} and that $\BDLat_{_\textup{FSI}}$ contains,  up to isomorphism, only the trivial and two-element bounded lattices. Clearly, $\BDLat_{_\textup{FSI}}$ has the \prp{1AP}, so $\BDLat$ has the \prp{AP}. Indeed, $\BLat$ and $\BDLat$ are the only non-trivial varieties of bounded lattices that have the \prp{AP}~\cite{DJ84}. On the other hand,  $\BDLat$, unlike $\BLat$, does not have the strong \prp{AP}. For  a counterexample, consider the sublattices $\m{A},\m{B},\m{C}\in\BDLat$ of $\m{M}_5$ (see Example~\ref{e:latticesnocep}) with $A=\set{\bot,a,\top}$, $B=\set{\bot,a,b,\top}$, and $C=\set{\bot,a,c,\top}$, and suppose that the doubly injective span $\tuple{\m{A},\m{B},\m{C},\f_B,\f_C}$ in $\BDLat$, with inclusion maps $\f_B$, $\f_C$, has an amalgam $\tuple{\m{D},\ps_B,\ps_C}$ in $\BDLat$. Then $\ps_B(b)$ and $\ps_C(c)$ are both complements of $\ps_B(a)=\ps_C(a)$ in $\m{D}$ and, by the uniqueness of complements in bounded distributive lattices, $\ps_B(b)=\ps_C(c)$. So $\ps_B[A] \neq \ps_B[B]\cap\ps_C[C]\neq\ps_C[A]$. \lipicsEnd
\end{example}

\begin{example}
A variety $\V$ of Heyting algebras generated by a finite totally ordered Heyting algebra $\m{H}$ has the \prp{AP} if, and only if, $n\coloneqq\lvert H\rvert\le 3$. Using J{\'o}nsson's Lemma for congruence-distributive varieties~\cite{Jon67}, the class $\Vfsi$ contains exactly $n$ algebras up to isomorphism. For $n \le 3$, it is clear that $\Vfsi$ has the  \prp{1AP} and hence $\V$ has the \prp{AP}. For $n>3$, we obtain a counterexample by considering $\m{A},\m{B},\m{C}\in\V$ such that $A =\set{\bot,a_1,\ldots,a_{n-3},\top}$, $B =\set{\bot,b,a_1,\ldots,a_{n-3},\top}$, and $C =\set{\bot,a_1,\ldots,a_{n-3},c,\top}$ with $\bot<b<a_1<\cdots<a_{n-3}<\top$ and $\bot<a_1<\cdots<a_{n-3}<c<\top$. If the doubly injective span $\tuple{\m{A},\m{B},\m{C},\f_B,\f_C}$, with inclusion maps $\f_B$, $\f_C$, were to have an amalgam $\tuple{\m{D},\ps_B,\ps_C}$ in $\V$, then $\V$ would contain the $n+1$-element totally ordered Heyting algebra generated by $\ps_B[B]\cup\ps_C[C]$, contradicting the fact that $\V\models\top\eq x_1\jn(x_1\ra x_2)\jn(x_2\ra x_3)\jn\cdots\jn(x_{n-1}\ra x_n)$.  \lipicsEnd
\end{example}

\begin{example}
An FL-algebra is said to be {\em semilinear} if it is isomorphic to a subdirect product of totally ordered FL-algebras. Such algebras provide algebraic semantics for broad families of many-valued logics (see,~e.g.,~\cite{MPT23}). In particular, {\em BL-algebras} --- algebraic semantics for H{\'a}jek's basic fuzzy logic $\lgc{BL}$ --- are term-equivalent to semilinear FL$_\text{e}$-algebras satisfying $\zr\le x$ and $x(x\ra y)\eq x\mt y$, while G{\"o}del algebras and MV-algebras are term-equivalent to BL-algebras satisfying $xx \eq x$ and $(x\to\zr)\to\zr \eq x$, respectively. 

If $\V$ is a variety of semilinear FL-algebras, then $\Vfsi$ consists of its totally ordered members, and if $\V$ also has the \prp{CEP}, then it has the \prp{AP} if, and only if, $\Vfsi$ has the \prp{1AP}. This correspondence has been used to establish the \prp{AP} or its failure for a wide range of varieties of semilinear FL-algebras (see, e.g.,~\cite{Mar12,MM12,MMT14,GJM20,FMS25,FG25,FS25a,FS25b,GU25}). For example, continuum-many varieties of semilinear FL-algebras satisfying $xx\eq x$ have the \prp{AP}, but only finitely many of these satisfy $xy\eq yx$~\cite{FMS25}. A full description of the varieties of BL-algebras that have the \prp{AP} has been given in~\cite{FS25a}; these include all varieties (and no more) of MV-algebras generated by a single totally-ordered algebra~\cite{DL00}, exactly three non-trivial varieties of G{\"o}del algebras (Boolean algebras, G{\"o}del algebras, and the variety generated by the three-element totally ordered Heyting algebra), 
and the variety of all BL-algebras~\cite{Mon06}. For further details, as well as proofs that the varieties of semilinear FL-algebras and semilinear FL$_\text{e}$-algebras do not have the \prp{AP}, we refer the reader to the survey article~\cite{FS25b}.  \lipicsEnd
\end{example}

\begin{remark}
Suppose that $\V$ is a finitely generated variety --- i.e., $\type$ is finite and $\V$ is generated as a variety by a given finite set of finite $\type$-algebras --- that is  congruence-distributive and such that $\Vfsi$ is closed under taking subalgebras. Then there are effective algorithms to decide if $\V$ has the \prp{CEP} or \prp{AP}. Using J{\'o}nsson's Lemma for congruence-distributive varieties~\cite{Jon67}, a finite set $\Vfsi^*\subseteq\Vfsi$ of finite algebras can be constructed such that each $\m{A}\in\Vfsi$ is isomorphic to some $\m{A}^*\in\Vfsi^*$. Hence, by Theorem~\ref{t:CEPvar}, it can be decided if $\V$ has the \prp{CEP} by checking if each member of $\Vfsi^*$ has the \prp{CEP}. Moreover, every finitely generated congruence-distributive variety that has the \prp{AP}  has the \prp{CEP}~\cite[Corollary~2.11]{Kea89}. Hence, by Theorem~\ref{t:APvar}, it can also be decided if $\V$ has the \prp{AP} by checking if $\Vfsi^*$ has the \prp{1AP}. \lipicsEnd
\end{remark}


\section{Deductive Interpolation Properties}\label{s:interpolation}

In this section, we consider a range of (deductive) interpolation properties that may or may not be possessed by a variety $\V$. These properties are all expressed in terms of equational consequence, or, equivalently,  via congruences of the free algebras of $\V$, and, like the Robinson property considered in Section~\ref{s:amalgamation}, transfer to various amalgamation properties of $\V$.

\begin{property}
A variety $\V$ has the {\em deductive interpolation property} (\prp{DIP}) if for any $\Si\subseteq\Eqc(\xbar,\ybar)$ and $\eps\in\Eqc(\ybar,\zbar)$ satisfying $\Si\mdl{\V}\eps$, there exists a $\Pi\subseteq\Eqc(\ybar)$ satisfying $\Si\mdl{\V}\Pi$ and $\Pi\mdl{\V}\eps$. 
\end{property}

\noindent
Note that, by the finitarity of equational consequence in a variety $\V$, we may assume that the `interpolant' $\Pi$ in this definition is finite. Observe also that $\V$ has the \prp{DIP} if, and only if, for any  $\Si\subseteq\Eqc(\xbar,\ybar)$, there exists a $\Ga\subseteq\Eqc(\ybar)$ (which cannot be assumed to be finite) satisfying $\Si\mdl{\V}\eps\iff\Ga\mdl{\V}\eps$ for any $\eps\in\Eqc(\ybar,\zbar)$. The right-to-left direction is immediate. For the left-to-right direction, given $\Si\subseteq\Eqc(\xbar,\ybar)$, define $\Ga\coloneqq\set{\de\in\Eqc(\ybar)\mid\Si\mdl{\V}\de}$ and consider any $\eps\in\Eqc(\ybar,\zbar)$. If $\Ga\mdl{\V}\eps$, then clearly $\Si\mdl{\V}\eps$. Conversely, if $\Si\mdl{\V}\eps$, then, by the \prp{DIP}, there exists a $\Pi\subseteq\Eqc(\ybar)$ satisfying $\Si\mdl{\V}\Pi$ and $\Pi\mdl{\V}\eps$. By definition, $\Pi\subseteq\Ga$, so  $\Ga\mdl{\V}\eps$.

The \prp{DIP} can also be formulated as a property of embeddings between congruence lattices of free algebras. Observe first that the inclusion map $\tau\colon\F(\ybar)\ra\F(\xbar,\ybar)$ `lifts' to the maps
\[
\begin{array}{rl}
\tau^*\colon\Con\F(\ybar)\ra\Con\F(\xbar,\ybar); &\The\mapsto\cg{\F(\xbar,\ybar)}(\tau[\The])\\[.05in]
\iv{\tau}\colon\Con\F(\xbar,\ybar)\ra\Con\F(\ybar); &\Psi\mapsto\iv{\tau}[\Psi] =\Psi\cap\textup{F}(\ybar)^2,
\end{array}
\]
yielding an {\em adjunction} $\tuple{\tau^*,\iv{\tau}}$; that is, for any $\The\in\Con\F(\ybar)$ and $\Psi\in\Con\F(\xbar,\ybar)$,
\[
\tau^*(\The)\subseteq\Psi\enspace\Longleftrightarrow\enspace\The\subseteq\iv{\tau}(\Psi).
\]
It follows, using  Lemma~\ref{l:bridge}, that $\V$  has the \prp{DIP} if, and only if, for any sets $\xbar,\ybar,\zbar$ with inclusion maps $\tau_1,\tau_2,\tau_3,\tau_4$ between free algebras of $\V$, the following diagram commutes:
\[
\begin{tikzpicture}[baseline=(current  bounding  box.center)]
 \matrix (m) [matrix of math nodes, row sep=1.5em, column sep=1.5em,ampersand replacement=\&]{
\Con \F(\xbar,\ybar) 		\& 		\& \Con  \F(\ybar) 		\\
 		\&		\& 		\\
\Con \F(\xbar,\ybar,\zbar) 	\&     		\& \Con \F(\ybar,\zbar) 	\\
};
\path
(m-1-1) edge[->] node[above] {$\iv{\tau}_1$} (m-1-3)
(m-1-1) edge[->] node[left] {$\tau_2^*$} (m-3-1)
(m-3-1) edge[->] node[below] {$\iv{\tau}_3$} (m-3-3)
(m-1-3) edge[->] node[right] {$\tau_4^*$} (m-3-3);
\end{tikzpicture}
\]
That is, the \prp{DIP} for $\V$ is equivalent to the `Beck-Chevalley-like' condition $\tau^*_4\iv{\tau}_1=\iv{\tau}_3\tau^*_2$ for appropriate inclusion maps $\tau_1,\tau_2,\tau_3,\tau_4$  between free algebras of $\V$. For the relationship of this condition to a version of the interpolation property formulated in categorical logic, we refer the reader to~\cite{Pit83,Pav91,GZ02}.

We now establish a bridge theorem relating the \prp{AP} and \prp{DIP}, leaning heavily on the correspondence between the \prp{AP} and \prp{RP} provided by Theorem~\ref{t:AP<=>RP}.

\begin{theorem}\label{t:apdip}
Let $\V$ be any variety.
\begin{enumerate}[(a)]
\item	If $\V$ has the amalgamation property, then it has the deductive interpolation property.
\item If $\V$ has the deductive interpolation property and the extension property, then it has the amalgamation property.
\end{enumerate}
\end{theorem}
\begin{proof}
(a) Suppose that $\V$ has the \prp{AP} and hence, by Theorem~\ref{t:AP<=>RP}, the \prp{RP}. Consider any $\Si\subseteq\Eqc(\xbar,\ybar)$ and $\eps\in\Eqc(\ybar,\zbar)$ such that $\Si\mdl{\V}\eps$, and define $\Pi\coloneqq\set{\de\in\Eqc(\ybar)\mid\Si\mdl{\V}\de}$. Clearly, $\Si\mdl{\V}\Pi$ and, since  $\Si\cup\Pi\mdl{\V}\eps$,  the \prp{RP} yields $\Pi\mdl{\V}\eps$.

(b) Suppose that $\V$ has the \prp{DIP} and the \prp{EP}. By Theorem~\ref{t:AP<=>RP}, it suffices to show that $\V$ has the \prp{RP}, so consider any $\Si\subseteq\Eqc(\xbar,\ybar)$ and $\Pi\subseteq\Eqc(\ybar,\zbar)$ satisfying $\Si\mdl{\V}\de\iff\Pi\mdl{\V}\de$ for all $\de\in\Eqc(\ybar)$, and any $\eps\in\Eqc(\ybar,\zbar)$ such that  $\Si\cup\Pi\mdl{\V}\eps$. By the \prp{EP}, there exists a $\De\subseteq\Eqc(\ybar,\zbar)$ such that $\Si\mdl{\V}\De$ and $\De\cup\Pi\mdl{\V}\eps$. By the \prp{DIP}, there exists a $\Ga\subseteq\Eqc(\ybar)$ satisfying $\Si\mdl{\V}\de\iff\Ga\mdl{\V}\de$ for all $\de\in\Eqc(\ybar,\zbar)$. In particular, $\Si\der{\V}\Ga$, so, by assumption, $\Pi\der{\V}\Ga$. But also $\Ga\der{\V}\De$, so $\Pi\der{\V}\De$. Hence $\Pi\mdl{\V}\eps$ as required.
\end{proof}

\begin{example}\label{e:FL-algebrasAP}
A variety of FL-algebras that has the \prp{CEP} (e.g., any variety of FL$_\text{e}$-algebras) has the \prp{AP} if, and only if, it has the \prp{DIP}, by Theorem~\ref{t:apdip}. This bridge has been traversed successfully in both directions. For example, it was proved in~\cite{FS24} that continuum-many varieties of FL$_\text{e}$-algebras have the \prp{AP} and hence that continuum-many axiomatic extensions of $\lgc{FL_e}$ have the \prp{DIP}. The \prp{AP} was established for the varieties of Abelian $\ell$-groups in~\cite{Pie72} and MV-algebras in~\cite{Mun88}, yielding the \prp{DIP} for Abelian logic $\lgc{A}$ and \L ukasiewicz logic $\mathrmL$, respectively, but in~\cite{MMT14} the  \prp{DIP} was proved directly, yielding alternative proofs of the \prp{AP} for these varieties. Note also that the failure of the \prp{AP} for the variety of $\ell$-groups, established in~\cite{Pie72}, has been used to prove that many other varieties of FL-algebras lack this property~\cite{GLT15,MPT23}. \lipicsEnd
\end{example}

The conjunction of the \prp{DIP} and \prp{EP} for a variety $\V$ yields a strictly stronger property. 

\begin{property}
A variety $\V$ has the {\em Maehara interpolation property} ($\prp{MIP}$) if for any $\Si\subseteq\Eqc(\xbar,\ybar)$ and $\Pi\cup\set{\eps}\subseteq\Eqc(\ybar,\zbar)$ satisfying  $\Si\cup\Pi\mdl{\V}\eps$, there exists a $\De\subseteq\Eqc(\ybar)$ satisfying $\Si\mdl{\V}\De$ and $\De\cup\Pi\mdl{\V}\eps$. 
\end{property}

The algebraic analogue of the  \prp{MIP} is a well-known categorical property.

\begin{property}
A class $\K$ of $\type$-algebras has the {\em transferable injections property} (\prp{TIP}) if for any injective span $\tuple{\m{A},\m{B},\m{C},\f_B,\f_C}$ in $\K$, there exist a $\m{D}\in \K$, a homomorphism $\ps_B\colon\m{B}\ra\m{D}$, and an embedding $\ps_C\colon\m{C}\ra\m{D}$ such that $\ps_B \f_B=\ps_C\f_C$ (see Figure~\ref{f:properties}(iv)).
\end{property}

The equivalence of the \prp{MIP} and \prp{TIP} for varieties was proved in~\cite{Ban69} (see also \cite{Wro84a}), and their equivalence to the conjunctions of the \prp{DIP} and \prp{CEP}, and the \prp{RP} and \prp{EP}, were established in~\cite{Cze85} and~\cite{Ono86}, respectively.

\begin{theorem}
The following statements are equivalent for any variety $\V$:
\begin{enumerate}[(1)]
\item$\V$ has the Maehara interpolation property.
\item$\V$ has the deductive interpolation property and the extension property.
\item $\V$ has the amalgamation property and the congruence extension property.
\item$\V$ has the transferable injections property.
\end{enumerate}
\end{theorem}

This leaves open the problem of describing an algebraic property that corresponds directly to the \prp{DIP}, to which (at least) two solutions may be found in the literature. A variety that admits free products has the \prp{DIP} if, and only if, it has the `flat amalgamation property'~\cite{Bac75}. More generally, a variety has the \prp{DIP} if, and only if, it has `diamond diagrams for embeddings'  (\prp{DDE}), studied in~\cite{KO10} as the ‘injective generalized amalgamation property’ for substructural logics and considered in a more general algebraic setting in~\cite{CG14}. 

\begin{example}
It was observed in~\cite{Ono86} that the variety of monoids has the flat amalgamation property, and it is easy to see that the same is true of the variety of commutative monoids, yielding examples of varieties that have the \prp{DIP} but not the \prp{AP} (see Example~\ref{e:monoidsAP}) or \prp{CEP}.  \lipicsEnd
\end{example}

\noindent
For the reader's convenience, some of the bridges between algebraic and syntactic properties presented (so far) in this chapter  are displayed in Figure~\ref{f:relationships}. 

\begin{figure}
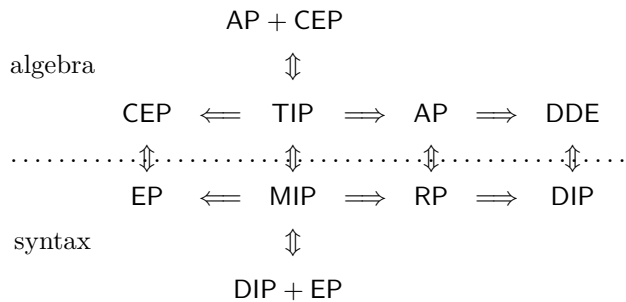

\begin{align*}
\begin{array}{ccccccccccc}
&&\multicolumn{3}{l}{\quad \prp{AP}+ \prp{CEP}}\\[.05in]
\text{algebra} && &\Updownarrow\\[.075in]
&\prp{CEP}  &\Longleftarrow & \prp{TIP}  &\Longrightarrow &  \prp{AP}  &\Longrightarrow &  \prp{DDE}\\[.05in]
 &\Updownarrow &  &\Updownarrow & &\Updownarrow &  &\Updownarrow &\\[-.2in]
 \hdotsfor{9}\\[.05in]
&\prp{EP}  &\Longleftarrow & \prp{MIP}  &\Longrightarrow &  \prp{RP}  &\Longrightarrow &  \prp{DIP}\\[.05in]
\text{syntax} && &\Updownarrow\\[.075in]
&& \multicolumn{3}{l}{\quad\; \prp{DIP}+ \prp{EP}}\\[.05in]
\end{array}
\end{align*}
\caption{Bridges between algebraic and syntactic properties}
\label{f:relationships}
\end{figure}

We conclude this section by considering a deductive version of uniform interpolation that is closely related to the \prp{DIP}, referring the reader to \refchapter{chapter:uniform} for a more nuanced account.

\begin{property}
A variety $\V$  has the {\em right uniform deductive interpolation property}  (\prp{RUDIP}) if for any finite $\Si\subseteq\Eqc(\xbar,\ybar)$, there exists a finite $\Ga\subseteq\Eqc(\ybar)$ satisfying $\Si\mdl{\V}\eps\iff\Ga\mdl{\V}\eps$ for any $\eps\in\Eqc(\ybar,\zbar)$. 
\end{property}

\noindent
Omitting the requirement that $\Ga$ be finite in this property yields the \prp{DIP}. Conversely, the \prp{RUDIP} may be viewed as the conjunction of the \prp{DIP} together with a further well-studied algebraic property. Recall that an algebra $\m{A}$ belonging to a variety $\V$ is {\em finitely presented in $\V$} if it is isomorphic to $\F(\xbar)/\The$ for some finite set $\xbar$ and finitely generated $\The\in\Con\F(\xbar)$. 

\begin{property}
A variety $\V$ is {\em coherent} if every finitely generated subalgebra of a finitely presented member of $\V$ is finitely presented.
\end{property}

\noindent
The notion of coherence originated in sheaf theory and has been studied broadly in the setting of groups, rings, modules, monoids, and other algebras (see, e.g.,~\cite{ES70,Gou92,KM18,KM19}). It has also been considered from a model-theoretic perspective by Wheeler~\cite{Whe76,Whe78}, who proved that the first-order theory of a variety has a model completion if, and only if, it is coherent and has the \prp{AP} and a further (rather complicated) property (see also~\cite{GZ02,vGMT17,MR23}). The connection to the \prp{RUDIP} is clarified by the following result.

\begin{theorem}[\cite{vGMT17,KM19}]
The following statements are equivalent for any variety $\V$:
\begin{enumerate}[(1)]
\item$\V$ is coherent.
\item For any finite sets $\xbar,\ybar$ and finitely generated congruence $\The$ of $\F(\xbar,\ybar)$, the congruence $\The\cap\Fc(\ybar)^2$ of $\F(\ybar)$ is finitely generated.
\item For any finite $\Si\subseteq\Eqc(\xbar,\ybar)$, there exists a finite $\Ga\subseteq\Eqc(\ybar)$ satisfying   $\Si\mdl{\V}\eps\iff\Ga\mdl{\V}\eps$ for any $\eps\in\Eqc(\ybar)$. 
\end{enumerate}
\end{theorem}

\begin{corollary}
A variety $\V$ has the right uniform deductive interpolation property if, and only if, it is coherent and has the deductive interpolation property.
\end{corollary}

The \prp{RUDIP} and a left uniform deductive interpolation property were established for the variety of Heyting algebras in~\cite{Pit92} and subsequently for many other varieties serving as algebraic semantics for propositional logics (see~\refchapter{chapter:uniform} for further details). Note, however, that although the variety of modal algebras has an implication-based uniform interpolation property, it is not coherent and therefore lacks the deductive version. Indeed, as shown in~\cite{KM18,KM19} using a general criterion, broad families of varieties providing algebraic semantics for modal and substructural logics fail to be coherent, and hence do not have the \prp{RUDIP}.


\section{The Craig Interpolation Property}\label{s:craig}

In this final section, we turn our attention to Craig interpolation, providing an equational formulation of this property for varieties of algebras with a lattice reduct and exploring its relationship to the \prp{DIP} and various forms of amalgamation.

Let us assume that $\type$ is an algebraic language with binary operation symbols $\mt$ and $\jn$ and, as before, at least one constant symbol, and write $s \le t$ as shorthand for $s\mt t\eq s$. We call an $\type$-algebra $\m{A}$ {\em lattice-ordered} if its reduct $\tuple{A,\mt,\jn}$ is a lattice. 

\begin{property}
A variety $\V$ of lattice-ordered $\type$-algebras has the {\em Craig interpolation property} (\prp{CIP}) if for any $\a\in\Fmc(\xbar,\ybar)$ and $\b\in\Fmc(\ybar,\zbar)$ satisfying $\mdl{\V}\a\le\b$, there exists a $\ga\in\Fmc(\ybar)$ satisfying $\mdl{\V}\a\le\ga$ and $\mdl{\V}\ga\le\b$. 
\end{property}

\noindent
Translating between consequence in propositional logics and their algebraic semantics as described in Section~\ref{s:algebra_and_logic}, this equational formulation of the \prp{CIP} corresponds directly for superintuitionistic, modal, and substructural logics to familiar formulations of the \prp{CIP} using an `implication' connective $\ra$ or $\ld$ (see \refchapter{chapter:nonclassical}).

Relationships between the \prp{CIP} and \prp{DIP} and various amalgamation properties depend heavily on the existence of suitable local deduction theorems. It follows easily from the deduction theorem for superintuitionistic logics that the \prp{CIP} and \prp{DIP}, and hence also the~\prp{AP}, coincide for any variety of Heyting algebras. As remarked in the introduction, there are precisely eight such varieties, corresponding to eight superintuitionistic logics, that have these properties~\cite{Mak77}. In general, however, neither the \prp{DIP} nor the \prp{AP} implies the \prp{CIP}. For example, the variety of MV-algebras has the  \prp{DIP} and \prp{AP}, but lacks the \prp{CIP}; e.g., $x\mt\lnot x\le y\lor \lnot y$ is satisfied by all MV-algebras, but every variable-free formula  is equivalent to $\bot$ or $\top$ in this variety, and both $x\mt\lnot x\le\bot$ and  $\top\le y\lor \lnot y$ fail in totally ordered MV-algebras with at least three elements. The failure of the \prp{CIP} for many other varieties of semilinear FL-algebras can be established similarly~\cite{MM12}. In particular, only three varieties of BL-algebras have this property: Boolean algebras, G{\"o}del algebras, and the variety generated by the three-element totally ordered Heyting algebra~\cite{Mon06}.

The following two propositions identify varieties providing algebraic semantics for broad families of substructural logics and modal logics for which the \prp{CIP} implies the \prp{DIP} and \prp{AP}.

\begin{proposition}\label{p:CIPimpliesDIP}
If a variety of FL$_\text{e}$-algebras has the Craig interpolation property, then it has the deductive interpolation property and amalgamation property.
\end{proposition}
\begin{proof}
Let $\V$ be a variety of FL$_\text{e}$-algebras that has the \prp{CIP}. It suffices to show that $\V$ has the \prp{DIP} (see Example~\ref{e:FL-algebrasAP}), so consider without loss of generality any  $\a\in\Fmc(\xbar,\ybar)$ and $\b\in\Fmc(\ybar,\zbar)$ satisfying $\set{\e\le\a}\mdl{\V}\e\le\b$. By the local deduction theorem  for varieties of FL$_\text{e}$-algebras (see Example~\ref{ex:FLeEP}), $\mdl{\V} (\a\mt\e)^n\le\b$, for some $n\in\N$. Hence, by the \prp{CIP}, there exists  a $\ga\in\Fmc(\ybar)$ satisfying  $\mdl{\V} (\a\mt\e)^n \le \ga$ and $\mdl{\V}\ga \le \b$. Using the local deduction theorem again twice, $\set{\e\le\a}\mdl{\V}\e\le\ga$ and $\set{\e\le\ga}\mdl{\V}\e\le b$. So $\V$ has the \prp{DIP}. 
\end{proof}

\begin{proposition}
If a variety of modal algebras has the Craig interpolation property, then it has the deductive interpolation property and amalgamation property.
\end{proposition}
\begin{proof}
Analogous to the proof of Proposition~\ref{p:CIPimpliesDIP}.
\end{proof}

Establishing the \prp{CIP} --- e.g., using the proof-theoretic methods explained in \refchapter{chapter:prooftheory} --- is therefore one way of proving that a variety of FL$_\text{e}$-algebras or modal algebras has the \prp{AP}. Indeed, the only known proof that the variety of FL$_\text{e}$-algebras has the \prp{AP} takes this approach (see~\cite{MPT23} or~\refchapter{chapter:prooftheory}). Conversely, to prove the \prp{CIP} for varieties of FL$_\text{e}$-algebras or modal algebras using algebraic methods, we require a strengthening of the \prp{AP}.

\begin{property}
A class $\K$ of lattice-ordered $\type$-algebras has the {\em superamalgamation property} (super \prp{AP}) if every doubly injective span $\tuple{\m{A},\m{B}_1,\m{B}_2,\f_1,\f_2}$ in $\K$ has an amalgam $\tuple{\m{C},\ps_1,\ps_2}$ in $\K$ satisfying  for any $b_1\in B_1$, $b_2\in B_2$ and distinct $i,j\in\set{1,2}$,
\begin{align*}
\ps_i(b_i)\le \ps_j(b_j) \enspace\Longrightarrow\enspace \ps_i(b_i)\le \ps_i \f_i(a) = \ps_j\f_j(a)  \le\ps_j(b_j)\,\text{ for some }a\in A.
\end{align*}
\end{property}

\noindent
It is not hard to see that if a variety of lattice-ordered $\type$-algebras has the super \prp{AP}, then it must also have the strong \prp{AP}. 

Bridge theorems between the \prp{CIP} and super \prp{AP} for varieties of modal algebras and FL$_\text{e}$-algebras are established similarly to Theorem~\ref{t:apdip}  relating the \prp{DIP} and~\prp{AP}.

\begin{theorem}[\cite{Mak79}]
A variety of modal algebras has the Craig interpolation property if, and only if, it has the superamalgamation property.
\end{theorem}


\begin{example}
\refchapter{chapter:modal} presents six quite different proofs that the variety of modal algebras --- the algebraic semantics of the modal logic $\lgc{K}$ --- has the \prp{CIP} and hence the super \prp{AP}, \prp{DIP}, and \prp{AP}, including one that establishes the \prp{CIP} via an analytic sequent calculus for $\lgc{K}$ and another that establishes the super \prp{AP} by constructing super-amalgams of spans of modal algebras. More generally, interpolation properties have been established for modal logics using a wide range of methods, both syntactic, using sometimes quite complex proof systems (see \refchapter{chapter:prooftheory}), and semantic (see, e.g.,~\cite{Mak91,CZ96,GM05}).   \lipicsEnd
\end{example}

\begin{theorem}[\cite{GJKO07}]
A variety of FL$_\text{e}$-algebras has the Craig interpolation property if, and only if, it has the superamalgamation property.
\end{theorem}

\begin{example}
A proof of the \prp{CIP} for the variety of FL-algebras --- the algebraic semantics of the full Lambek calculus $\lgc{FL}$ --- is obtained using an analytic sequent calculus (see,~e.g.,~\cite{GJKO07,MPT23}). This variety does not have the \prp{AP} (see~\cite{JS25}, also for failures of the \prp{AP} for related varieties), although the question of whether it has the \prp{DIP} is still open. The varieties of $\ell$-groups and semilinear residuated lattices do not have the \prp{CIP} or the \prp{AP} (see Example~\ref{e:FL-algebrasAP}), but the status of the \prp{DIP} is open also for these cases.  \lipicsEnd
\end{example}

\noindent
Table~\ref{t:amalg} displays the status of the \prp{CEP}, \prp{CIP}, \prp{DIP}, and \prp{AP} for some of the varieties featured in this chapter (where `n/a' stands for `not applicable'  and `?' stands for `open problem').\footnote{Note that although Abelian groups lack a lattice reduct, they satisfy the Craig interpolation property with respect to an implication defined by $a\to b\coloneqq \iv{a}b$.}

\begin{table}
\begin{center}
\begin{tabular}[t]{lllll} 
\hline
Variety    						& 	\prp{CEP} &	\prp{CIP}	& \prp{DIP}&	 \prp{AP}\\
\hline         
Boolean algebras				&	yes		&	yes		&	yes		&	yes\\
modal algebras					&	yes		&	yes		&	yes		&	yes\\
Heyting algebras				&	yes		&	yes		&	yes		&	yes\\
G{\"o}del algebras 				&	yes		&	yes		&	yes		&	yes\\
MV-algebras					&	yes		&	no		&	yes		&	yes\\
BL-algebras					&	yes		&	no		&	yes		&	yes\\
FL-algebras 					&	no		&	yes		&	?		&	no\\
FL$_\text{e}$-algebras			&	yes		&	yes		&	yes		&	yes\\
semilinear FL-algebras 			&	no		&	no		&	?		&	no\\
semilinear FL$_\text{e}$-algebras	&	yes		&	no		&	no		&	no\\
$\ell$-groups					&	no		&	no		&	?		&	no\\
Abelian $\ell$-groups			&	yes		&	no		&	yes		&	yes\\
bounded lattices				&	no		&	yes		&	yes		&	yes\\
bounded distributive lattices		&	yes		&	yes		&	yes		&	yes\\
commutative monoids			&	no		&	n/a		&	yes		&	no\\
Abelian groups					&	yes		&	n/a	&	yes		&	yes\\
\hline
\end{tabular} 
\end{center}
\caption{Interpolation and amalgamation properties for a selection of varieties} 
\label{t:amalg}
\end{table}

\section*{Acknowledgments}
\addcontentsline{toc}{section}{Acknowledgments}

I would like to thank Wesley Fussner, Sam van Gool, Isabel Hortelano-Martin, and Simon Santschi for their careful reading of previous versions of this chapter and helpful feedback. I also acknowledge support from the Swiss National Science Foundation (SNSF) grant no. 200021\textunderscore215157.

 
\bibliography{bibliography}
 
\end{document}